\documentclass[a4paper,11pt]{article}
\usepackage[T1]{fontenc}
\usepackage[utf8x]{inputenc} %extra unicode
\usepackage{lmodern}
\usepackage{amsmath}
\usepackage{amsthm}
\usepackage{amssymb}
\usepackage{authblk}
\usepackage{graphicx}
\usepackage{enumitem}
\usepackage{xcolor}
\usepackage[draft=true]{hyperref}

\usepackage{amsthm}
\usepackage{amssymb}
\usepackage{xparse}
\usepackage{thmtools}
\usepackage{stackrel}
\usepackage{bbold}
\usepackage{relsize}
\usepackage{tikz}
\usepackage{dsfont}
\usepackage{bigints}

\DeclareFontFamily{OMX}{lmex}{}
\DeclareFontShape{OMX}{lmex}{m}{n}{<-> lmex10}{}  %fixing
\usetikzlibrary{hobby}
\newcommand{\R}{\mathbb{R}}
\newcommand{\C}{\mathbb{C}}

\newcommand{\N}{\mathbb{N}}

\newcommand{\dd}{\mathrm{d}}

\renewcommand{\S}{\mathbb{S}}

\DeclareMathOperator{\supp}{supp}

\DeclareMathOperator{\dist}{dist}

\makeatletter
\newtheoremstyle{indented}
{7pt} %vertical space before
{7pt} % vertical space after
{} %{\addtolength{\@totalleftmargin}{2.5em}
{1.5em} %indent
{\bfseries} %header font
{.} %punctuation
{.5em} %horizontal space after header
{} %header specification

\theoremstyle{definition}

\newtheorem{defn}{Definition}[section]

\theoremstyle{plain}
\newtheorem*{theorem*}{Theorem}

\newtheorem{theorem}{Theorem}

\newtheorem{prop}[defn]{Proposition}

\newtheorem{lem}[defn]{Lemma}
\newtheorem{conj}{Conjecture}
\theoremstyle{definition}
\newtheorem{rem}[defn]{Remark} %remarks are not indented

\makeatletter
\renewcommand*\env@matrix[1][*\c@MaxMatrixCols c]{%
  \hskip -\arraycolsep
  \let\@ifnextchar\new@ifnextchar
  \array{#1}}
\makeatother

\usepackage{a4wide}
\title{Fractional exponential decay in the forbidden region for
  Toeplitz operators}
\author{Alix Deleporte\thanks{alix.deleporte@math.uzh.zh}}
\affil{University of Z\"urich, Institute for
  Mathematics\\Winterthurerstrasse 190, CH-8057 Z\"urich}

\usepackage{bbm}

\usetikzlibrary{decorations.pathreplacing}
\usetikzlibrary{decorations.markings}

\begin{document}

\maketitle

%\blfootnote{This work was supported by grant
%  ...\\
%  MSC 2010 Subject classification: ...}
\begin{abstract}
  We prove several results of concentration for eigenfunctions in
  Toeplitz quantization. With mild assumptions on the regularity, we
  prove that eigenfunctions are $O(\exp(-cN^{\delta}))$ away from the
  corresponding level set of the symbol, where $N$ is the inverse
  semiclassical parameter and $0<\delta<1$ depends on the regularity. As an application, we
  prove a precise bound for the free energy of spin systems at high temperatures,
  sharpening a result of Lieb.
\end{abstract}

\section{Introduction}
\label{sec:introduction}

Localisation or microlocalisation estimates are central in
semiclassical analysis.
The most practical context for studying localisation of quantum states
is the case of a smooth symbol on a fixed, finite-dimensional
manifold. Indeed, in this case one can use the symbolic calculus to
prove $O(\hbar^{\infty})$ decay in the forbidden region.

How to improve these bounds? One idea is to impose
more regularity (e.g. real-analyticity) and try to obtain more precise
microlocalisation estimates (see section 3.5 in
\cite{martinez_introduction_2002} for the pseudodifferential case and
\cite{deleporte_toeplitz_2018} for the Toeplitz case). Among
recent work developping or using exponential estimates in analytic regularity, one can cite magnetic Schrödinger
operators \cite{bonthonneau_wkb_2017,bonthonneau_exponential_2019},
the focusing NLS equation \cite{fujiie_semiclassical_2019},
resonances of Schrödinger operators \cite{fujiie_width_2011} and the
Steklov problem \cite{galkowski_pointwise_2018}.

In this article, we are interested in localisation estimates in low
regularity for
\emph{Toeplitz quantization} \cite{le_floch_brief_2018}. Given a compact Kähler manifold
$(M,\omega,J)$, where $\omega$ is a symplectic form with integer
periods and $J$ is a
complex structure, one can construct a Hermitian complex line bundle
$(L,h)$ over $M$, such that $\text{curv}(h)=2i\pi\omega$; then the essential
ingredient for the quantization is the family of \emph{Szeg\H{o}
  projectors} $(S_N)_{N\in \N}$: for every $N\in \N$, $S_N$ is the
orthogonal projector from the section space $L^2(M,L^{\otimes N})$ to
the subspace of holomorphic sections $H^0(M,L^{\otimes N})$. Then, the
Toeplitz operator $T_N(f)$ associated with a function $f:M\to \C$ is
the composition of the multiplication by $f$ and the Szeg\H{o} projector:
\begin{center}
\begin{tabular}{cccc}
  $T_N(f):$&$H^0(M,L^{\otimes N})$&$\rightarrow $&$H^0(M,L^{\otimes N})$\\
  &$u$&$\mapsto$&$S_N(fu).$ 
\end{tabular}
\end{center}

One should think of $N$ as an inverse semiclsasical parameter:
$N=\hbar^{-1}$. The Toeplitz operator $T_N(f)$ is well-defined, and uniformly bounded
in operator norm, as long as $f\in L^{\infty}$. This fact already hints
towards a different behaviour of Toeplitz and Weyl quantization for
low-regularity symbols (in Weyl quantization, one must assume some
regularity to obtain $L^2\to L^2$ boundedness).

We are now ready to state the first main result of this article.

\begin{theorem}\label{thr:conc-Linf}
  Let $(M,\omega,J)$ be a compact, quantizable Kähler manifold. Let
  $\alpha=\frac 12$ if $(M,\omega,J)$ is $C^{1,1}$ and $\alpha=1$ if
  $(M,\omega,J)$ is real-analytic.
  
  Let $f\in L^{\infty}(M,\R)$. For every $\delta>0$ there exist $C>0,
  c>0, N_0>0$ such that, for any $N\geq N_0$, for any $\epsilon>CN^{-\frac 14+\delta}$, for any normalised
  $u\in H^0(M,L^{\otimes N})$ and any $\lambda\in \R$ such that
  \[
    T_N(f)u=\lambda u,
  \]
  with \[W=\{x\in M, \dist(x,\{f\geq \lambda+\epsilon\})>\epsilon\},\] one has
  \begin{equation*}
    \|u\|_{L^2(W)}^2\leq \frac{C}{\epsilon}\exp(-c(N\epsilon^4)^{\frac{\alpha}{2\alpha+1}}).
  \end{equation*}
\end{theorem}
In particular, if $W$ is at fixed distance from a sublevel of $f$ (that is, if
$\epsilon$ does not depend on $N$), then the mass of $u$ on $W$ is
always $O(\exp(-cN^{\frac 14}))$. This precision is much better than
the symbolic calculus even for smooth symbols on smooth manifolds
(which only leads to $O(N^{-\infty})$) and, in fact, it is more
precise than the knowledge of the Szeg\H{o} projector.

In fact, Theorem \ref{thr:conc-Linf}, as well as Theorems
\ref{thr:conc-Lip} and \ref{thr:conc-C2-low}, only depend on the
off-diagonal decay of the Szeg\H{o} projector (Proposition
\ref{prop:kernel-decay}). In particular, equivalents of these Theorems
hold on various generalisations of Kähler quantization, as long as
this off-diagonal decay holds: spin$^c$-Dirac quantization
\cite{ma_holomorphic_2007}, or Bochner Laplacians
\cite{guillemin_laplace_1988,kordyukov_generalized_2019}. Semiclassical
constructions of quantizations, like the one used for almost Kähler
quantization (appendix of \cite{boutet_de_monvel_spectral_1981}) do
are not precise enough here: they are only defined modulo $O(N^{-\infty})$ so the
kernel decay is blurred at this limit. However, all methods used here
work in this context, yielding $O(N^{-\infty})$ estimates for
low-regularity symbols.

The factor $N\epsilon^4$, or equivalently the condition
$\epsilon>CN^{-\frac 14+\delta}$, does not correspond to usual statements
about microlocalisation. Usually, operator calculus works for symbols
in mildly exotic classes $S_{\frac 12 -\delta}$, so that one can prove
$O_{\delta}(\hbar^{\infty})$ decay at distance $\hbar^{\frac 12 -
  \delta}$.

The FBI transform (or equivalently, the Bargmann transform) allows to
conjugate Toeplitz operators on $\C^n$ with pseudodifferential
operators on $\R^{2n}$. Unfortunately, the error terms in this
conjugation are usually much larger than the decay rates in Theorem
\ref{thr:conc-Linf}: indeed, even for $C^{\infty}$ symbols it is not better than
$O(\hbar^{\infty})$. Thus one cannot apply Theorem \ref{thr:conc-Linf}
to
pseudodifferential operators. Apart from the case of Gevrey or
analytic regularity, the only situation in which one is able prove
exponential decay for pseudodifferential operators is Agmon estimates
for differential operators \cite{agmon_lectures_2014}.

On the Toeplitz side, the quantization of indicator function of sets
has raised recent interest \cite{charles_entanglement_2018,zelditch_central_2019}, in connection with Fermi
statistics. We also must mention the work \cite{kordyukov_semiclassical_2018}, which obtains fractional
exponential decay (more precisely, $O(\exp(-cN^{\frac 12}))$) at
finite distance for Toeplitz operators with $C^{\infty}$ symbols; in
fact, the proof of this localisation result only uses $C^{1,1}$ regularity of
the symbol. The method used is a weighted
estimate for the Kohn Laplacian (or rather, the Bochner Laplacian):
one writes $S_N$ as the kernel of an elliptic differential operator,
then conjugate with rapidly oscillating weights.

Using the decay properties of the Szeg\H{o} projector, one can
simplify a great deal the method used in
\cite{kordyukov_semiclassical_2018} and relax the regularity
hypotheses. This leads to the following improvement of Theorem
\ref{thr:conc-Linf}.

\begin{theorem}\label{thr:conc-Lip}
  Let $(M,\omega,J)$ be a compact, quantizable Kähler manifold of
  regularity $C^{1,1}$.
  
  Let $f\in \mathrm{Lip}(M,\R)$. There exist $C>0,
  c>0$ such that, for any $N\in \N$, for any $\epsilon>CN^{-\frac 12}$, for any normalised
  $u\in H^0(M,L^{\otimes N})$ and any $\lambda\in \R$ such that
  \[
    T_N(f)u=\lambda u,
  \]
  if \[W=\{x\in M, \dist(x,\{f\geq \lambda+CN^{-\frac 12}\})>\epsilon\},\] one has
  \begin{equation*}
    \|u\|_{L^2(W)}^2\leq CN^{\frac 12}\exp(-c\epsilon \sqrt{N}).
  \end{equation*}
\end{theorem}
A byproduct of Theorem \ref{thr:conc-Lip} is that the eigenfunction
$u$ is $O(N^{\infty})$ (in fact, exponentially small) on
$\{|f-\lambda|>N^{-\frac 12+\delta}\}$, for any $\delta>0$. If
$\lambda$ is a regular value of $f$, the sharpness of this
localisation region cannot be improved: the uncertainty principle
forbids quantum states in Toeplitz quantization to be concentrated on
a band thinner than $N^{-\frac 12}$.

A version of Theorem \ref{thr:conc-Lip} is used in
\cite{kordyukov_semiclassical_2018} to study the low-energy spectrum
of symbols with more regularity. If $f\in C^{1,1}(M,\R)$ and $\min(f)=0$, then testing
against coherent states shows that the smallest eigenvalue of $T_N(f)$
is of order
$\min(\mathrm{Sp}(T_N(f)))=O(N^{-1})$. In this situation, one should
expect the corresponding eigenvector $u$ to be concentrated on
$\{f\leq N^{-1+\delta}\}$. In the case where $f\in C^{\infty}$, this
can be obtained from the symbolic calculus \cite{charles_quantization_2016,deleporte_low-energy_2017}. Here, we are able
to modify the proof of Theorem \ref{thr:conc-Lip}, yielding a sharper
result.

\begin{theorem}\label{thr:conc-C2-low}
  Let $(M,\omega,J)$ be a compact, quantizable K\"ahler manifold of
  regularity $C^{1,1}$.

  Let $f\in C^{1,1}(M,\R)$ with $\min(f)=0$. For every $\delta>0$ and every $C_0>0$, there exists $C>0$ and
  $c>0$ such that, for any $N\in \N$, for any normalixed $u\in
  H^0(M,L^{\otimes N})$ and any $\lambda<C_0N^{-1}$ such that
  \[
    T_N(f)u=\lambda u,
  \]
  one has
  \[
    \|u\|^2_{L^2(\{f\geq N^{-1+\delta}\})}\leq
    Ce^{-cN^{\frac{\delta}{2}}}.
    \]
  \end{theorem}

A natural set of quantum Hamiltonians which can be written as Toeplitz
operators consists in \emph{spin operators}: here, the manifold is
$(\C\mathbb{P}^1)^d\approx (\S^2)^d$, and the symbol $f$ is a
polynomial in the coordinates for the natural immersion into
$(\R^3)^d$. Such a symbol is real-analytic, so for fixed $d$ and $N\to
+\infty$ this result is weaker than the $O(\exp(-cN))$ decay
established in previous work \cite{deleporte_toeplitz_2018}. However, in experimental
situations $d$ is much larger than $N$, which raises the question of
uniform (in $d$) localisation estimates for a reasonable sequence of
symbols.

Usual
tools for the study of microlocalisation fail in this context.
The
symbolic calclulus makes sense for fixed $d$ but goes awry as $d$ increases: 
for instance, the stationary phase lemma typically requires a number of derivatives
which grows linearly with $d$. Theorems \ref{thr:conc-Lip} and
\ref{thr:conc-C2-low} rely on the pointwise decay property of the
Szeg\H{o} projector by means of the Schur test. This also
fails in large dimension (see Subsection \ref{sec:products-spheres}).

However, the method of proof used in
\cite{kordyukov_semiclassical_2018} adapts to the limit $d\to +\infty$
quite well. Controlling the various constants yields
\begin{theorem}\label{thr:weighted-spins}
  Let $g$ be a tame spin system (see Definition \ref{def:tame-spin-system}). There exists $C>0$ and $c>0$ such
  that, for every $N\in \N$, for every $d\geq d_0(N)$ large enough, for every $u\in
  H^0((\S^2)^d,L^{\otimes N})$ of norm $1$ and $\lambda\in \R$ such
  that
\[
T_N(g)u=\lambda u,
\]
then with
\[
U=\{|g-\lambda|<CN^{-\frac 14}d^{\frac
  34}\}
\]
and
\[
W=\{x\in (\S^2)^d,\dist(x,U)>CN^{-\frac 12}\sqrt{d}\},
\]
one has
\[
\int_W e^{c\sqrt{N}\frac{\dist(x,U)}{\sqrt{d}}}|u(x)|^2\leq C.
\]
\end{theorem}

Localisation estimates can be used to understand, at least at dominant
order, the behaviour of the heat operator generated by
$T_N(f)$. This heat operator is the complex extension of the wave
propagator, restricted to imaginary time. The analysis of
this operator is pertinent not only with respect to the Egorov
theorem, but also because it is believed to be related to geodesics in
the space of Kähler metrics on $M$. Furthermore, in the case of spin
systems, the quantity $Z=Tr(e^{-\beta T_N(f)})$ is called
\emph{partition function} at inverse temperature $\beta$ and is a key
element of the understanding of the statistical mechanics of spin
systems.

\begin{prop}\label{prop:control-free-energy}
  Let $g$ be a tame spin system. Consider, for $N\in \N$ and
  $\beta\geq 0$, the quantum free energy
  \[
    f_Q=-\frac{1}{\beta d}\log(\mathrm{Tr}(\exp(-\beta T_N(g)))).
  \]
  Consider also the normalized classical free energy
  \[
    f_C=-\frac{1}{\beta
      d}\log\left[\left(\frac{N+1}{\pi}\right)^{d}\int_{(\S^2)^d}e^{-\beta
        g}\right].
  \]
  Then there exists $c>0$ and $C>0$ such that, uniformly in $d$ and
  $N$, uniformly in $\beta\leq cN^{\frac 12}d^{-1}$, one has
  \[
    |f_C-f_Q|\leq CN^{-\frac 12}.
    \]
  \end{prop}
  As for the standard estimate found in \cite{lieb_classical_1973},
  Proposition \ref{prop:control-free-energy} is a ``Weyl-law'' type
  control: one estimates a quantum quantity, related to the
  distribution of eigenvalues, using only the volume form on the phase
  space. Such estimates cannot distinguish between situations where
  there is a phase space transformation preserving the volume form but
  not the symplectic form (for instance, between a Heisenberg
  antiferromagnet and a Heisenberg ferromagnet). 

This article is organised as follows. In Section
\ref{sec:rate-decay-szegho} we review the properties of the Szeg\H{o}
projector that we will use to prove Theorems \ref{thr:conc-Linf},
\ref{thr:conc-Lip} and \ref{thr:conc-C2-low}. In particular,
Subsection \ref{sec:products-spheres} is devoted to an analysis of the
case of a product of a large number of spheres.

In Section \ref{sec:fract-decay-eigenf}, we prove Theorem
\ref{thr:conc-Linf}. The method used is a decomposition of $M$ into
shells corresponding to the distance to a level set. In Section
\ref{sec:decay-eigenf-lipsch}, we derive weighted estimates by
simplifying the methods of \cite{kordyukov_semiclassical_2018}, in
order to prove Theorems \ref{thr:conc-Lip} and \ref{thr:conc-C2-low}.

The two last sections of this article are devoted to Theorem
\ref{thr:weighted-spins} and Proposition
\ref{prop:control-free-energy}. In Section
\ref{sec:weighted-estimates}, we review the proof of the weighted
estimate in \cite{kordyukov_semiclassical_2018}, and we give an
explicit dependence of the constants in the objects (the manifold, the
weight, and the symbol). In Section \ref{sec:case-study:-spin}, we
construct a weight adapted to a spin system in large dimension, and
conclude the proofs. 

%%% Local Variables:
%%% mode: latex
%%% TeX-master: "Article.tex"
%%% End:

\section{Rate of decay of the Szeg\H{o} projector}
\label{sec:rate-decay-szegho}

\subsection{General case}
\label{sec:general-case}

One of the essential properties of the Szeg\H{o} projector is its
rapid off-diagonal decay. It is much easier to derive a good
off-diagonal decay rate than to study the Szeg\H{o} projector near the
diagonal with a corresponding degree of precision; in fact, safe for
the case where $M$ is real-analytic, the off-diagonal decay is faster
than the precision available on the diagonal.

\begin{prop}[Pointwise estimates]\label{prop:pointwise-decay}
  Let $M$ be a compact Kähler quantizable manifold of complex dimension $d$. For
  $N\in \N$, let $S_N$ denote the Szeg\H{o} (or Bergman) projector on
  $M$. Then the following is true.
  \begin{enumerate}
  \item \cite{delin_pointwise_1998} If the metric of $M$ is $C^{1,1}$, then there exist $C>0, c>0$
    such that, for any $N\in \N$, for any $(x,y)\in M^2$,
    \[
      |S_N(x,y)|^2\leq CN^d\exp(-c\sqrt{N}\dist(x,y)).
    \]
    \item \cite{berman_direct_2008} If the metric of $M$ is real-analytic, then there
      exist $C>0$, $c>0$ such that, for any $N\in \N$, for any
      $(x,y)\in M^2$,
      \[
        |S_N(x,y)|^2\leq CN^d\exp(-cN\dist(x,y)^2).
      \]
  \end{enumerate}
\end{prop}
In the previous Proposition, the decay rate of case 1 is essentially
sharp (up to a power of $\log(N)$) if the metric of $M$ is $C^{\infty}$ or less \cite{christ_slow_2003}. Case 2 is
also sharp: in the easiest examples $M=\C^n$ or $M=\C\mathbb{P}^n$,
one has exactly $|S_N(x,y)|^2=CP(N)\exp(-cN(\dist(x,y)^2+O(\dist(x,y)^4)))$.
In the case of $s$-Gevrey regularity, one can interpolate between cases 1
and 2, obtaining $(N\dist(x,y)^2)^{\frac{s}{2s-1}}$, see
\cite{hezari_quantitative_2018}; we do not know if this decay rate is sharp.

This pointwise decay immediately leads, via the Schur test, to a
decay in terms of operators.

\begin{prop}[Operator estimates]\label{prop:kernel-decay}
  Let $M$ be a compact Kähler quantizable manifold of complex dimension $d$. For
  $N\in \N$, let $S_N$ denote the Szeg\H{o} (or Bergman) projector on
  $M$. Then the following is true.
  \begin{enumerate}
  \item If the metric of $M$ is $C^{1,1}$, then there exist $C>0, c>0$
    such that, for any $N\in \N$, for any open sets $U,V$ of $M$,
    \[
      \|\mathds{1}_US_N\mathds{1}_V\|_{L^2\mapsto L^2}\leq C\exp(-c\sqrt{N}\dist(U,V)).
    \]
    \item If the metric of $M$ is real-analytic, then there
      exist $C>0$, $c>0$ such that, for any $N\in \N$, for any
      open sets $U,V$ of $M$,
      \[
        \|\mathds{1}_US_N\mathds{1}_V\|_{L^2\mapsto L^2}\leq C\exp(-cN\dist(U,V)^2).
      \]
    \end{enumerate}
\end{prop}
The constant $C$ is not trivial to get rid of. In particular, one gets
estimates of the form
\[
  \|\mathds{1}_US_N\mathds{1}_V\|_{L^2\mapsto L^2}\leq
  \exp(-c'(N\dist(U,V)^2)^{1\text{ or }1/2})
\]
only under the condition that $\dist(U,V)\geq C_1N^{-\frac 12}$. This
remark is of little importance on a fixed Kähler manifold, but as we
will see, the constant $C$ blows up with the dimension in the case
$M=(\S^2)^d$, at least when using a Schur test.

\subsection{Products of spheres}
\label{sec:products-spheres}

This subsection is devoted to a discussion of Proposition \ref{prop:kernel-decay}
in the case $M=(\S^2)^d$. Unfortunately, we are not able to prove a
$d$-independent version of Proposition \ref{prop:kernel-decay} in this context, but we
conjecture it is the case, and give a simple proof of a weaker result.

We take the following scaling convention: the area
of the sphere is $1$.
The Szeg\H{o} kernel on $(\S^2)^d$ is easily obtained from
that on $\S^2$: one has
\[
  |S_{N,d}(x,y)|=(N+1)^d\prod_{x=1}^d(x_i\dot y_i)^N.
\]
For fixed $d$ and $x\neq y$, as $N\to +\infty$ this quantity decays
exponentially fast. As $d$ increases, however, this behaviour is
destroyed. It makes sense to try to estimate operator norms of the
form
\[
  \|\mathds{1}_US_N\mathds{1}_V\|_{L^2\to L^2}
\]
where $U$ and $V$ are at positive distance, independently on
$d$. Indeed, in this version of the kernel estimate the factor $N^d$
is not present anymore (see the difference between Propositions \ref{prop:pointwise-decay}
and \ref{prop:kernel-decay}). Moreover, in the proof of Theorem
\ref{thr:conc-Linf}, we only use Proposition \ref{prop:kernel-decay}.

\begin{lem}\label{lem:cos}
  For $0\leq \theta\leq \pi/2$ one has
  \[
    \cos(\theta)\leq \exp(-\theta^2/2).
    \]
  \end{lem}
  \begin{proof}
    The two first non-zero terms in a Taylor expansion on both sides
    coincide, so that
    \[
      \exp(-\theta^2/2)-\cos(\theta)=\sum_{k=2}^{+\infty}(-1)^k\theta^{2k}\left[\frac{1}{2^kk!}-\frac{1}{(2k)!}\right].
    \]
    The claim then follows from the fact that the non-negative sequence
    \[
      \left(\theta^{2k}\left[\frac{1}{2^kk!}-\frac{1}{(2k)!}\right]\right)_{k\geq 2}
    \]
    is non-increasing and the alternating
    series theorem.

    Indeed, the difference between two consecutive terms is
    \[
      \frac{\theta^{2k}}{2^kk!}\left[1-\frac{\theta}{2(k+1)}\right]-\frac{\theta^{2k}}{(2k)!}\left[1-\frac{\theta}{(2k+1)(2k+2)}\right].
    \]
    Since $\frac{\theta}{2(k+1)}\leq \frac{\pi}{12}$, the difference
    between two consecutive terms is larger than
    \[
      \frac{\theta^{2k}}{(2k)!}\left[1\cdot 3 \cdots
        (2k-1)\left(1-\frac{\pi}{12}\right)-1\right]
      \geq \frac{\theta^{2k}}{(2k)!}\left[2-\frac{\pi}{4}\right]\geq
      0.
    \]
  \end{proof}

  \begin{prop}\label{thr:decay-op}
    Let $d,N$ be positive 
    integers and let $D>0$.
    Let $U,V$ be subsets of $(\S^2)^d$ such that $\dist(U,V)=D>0$.

    Then
    \[
      \|\mathds{1}_US_N\mathds{1}_V\|_{L^1\to L^{\infty}}\leq
      \frac{4}{\sqrt{2\pi d}}4^d\exp(-(N+1)D^2/16).
    \]
    In particular,
    \[
      \|\mathds{1}_US_N\mathds{1}_V\|_{L^2\to L^2}\leq
      \frac{4}{\sqrt{2\pi d}}4^d\exp(-(N+1)D^2/16).
      \]
  \end{prop}
  \begin{proof}
    One has
    \[
      \|\mathds{1}_US_N\mathds{1}_V\|_{L^1\to L^{\infty}}=
      \sup_{x\in U}\int_{y\in V}|S_N(x,y)|.
    \]
    Letting $P=[0,\pi]^d$ and $B(0,D)$ denote the Euclidean ball of
    radius $D$ in $\R^d$, one has
    \[
      \|\mathds{1}_US_N\mathds{1}_V\|_{L^1\to L^{\infty}}
      \leq \frac{(N+1)^d}{2^d}\int_{P\setminus
        B(0,D)}\prod_{j=1}^d\underbrace{\cos(\theta_j/2)^N\sin(\theta_j)}_{=2\cos(\theta_j/2)^{N+1}\sin(\theta_j/2)}
      \dd \theta_1\cdots \dd\theta_d.
    \]
    From Lemma \ref{lem:cos} and the classic inequality $|\sin(x)|\leq x$, one
    is left with
      \[
      \|\mathds{1}_US_N\mathds{1}_V\|_{L^1\to L^{\infty}}
      \leq \frac{(N+1)^d}{2^d}\int_{P\setminus
        B(0,D)}e^{-(N+1)\theta^2/8}\left(\prod_{j=1}^d\theta_j\right)
      \dd \theta_1\cdots \dd\theta_d.
    \]
    Letting $\widetilde{P}=\{z\in \C,|z|<\pi\}^d$ and $\widetilde{B}(0,D)$
    denote the Hilbert ball of radius $D$ in $\C^d$, one has
    $\widetilde{P}\subset \widetilde{B}(0,\sqrt{d}\pi)$, so that
    \begin{align*}
      \int_{P\setminus
        B(0,D)}e^{-(N+1)\theta^2/8}\left(\prod_{j=1}^d\theta_j\right)
      \dd \theta_1\cdots \dd\theta_d&=\frac{1}{(2\pi)^d}\int_{\widetilde{P}\setminus
        \widetilde{B}(0,D)}e^{-(N+1)|z|^2/8}\dd z_1 \dd
      \overline{z_1}\ldots \dd z_d\dd \overline{z_d}\\ &\leq \frac{1}{(2\pi)^d}\int_{\widetilde{B}(0,\sqrt{d}\pi)\setminus
        \widetilde{B}(0,D)}e^{-(N+1)|z|^2/8}\dd z_1 \dd
                                                         \overline{z_1}\ldots \dd z_d\dd \overline{z_d}\\
      &=\frac{\omega_{2d-1}}{2(2\pi)^d}\int_{D^2}^{d\pi^2}e^{-(N+1)u/8}u^{d-1}\dd u.
    \end{align*}
    Here $\omega_{2d-1}=\frac{2\pi^d}{(d-1)!}$ is the volume of the unit sphere in dimension
    $2d-1$.

    The Stirling formula yields
    \begin{align*}
      \|\mathds{1}_US_N\mathds{1}_V\|_{L^1\to L^{\infty}}
      &\leq \frac{(N+1)^d}{\sqrt{2\pi d}}\int_{D^2}^{d\pi^2}e^{-(N+1)u/8}\left(\frac{eu}{4(d-1)}\right)^{d-1}\dd
        u\\
      &=\frac{1}{\sqrt{2\pi d}}\int_{D^2(N+1)}^{d\pi^2(N+1)}e^{-x/8}\left(\frac{ex}{4(d-1)}\right)^{d-1}\dd x.
    \end{align*}
    The quantity to be integrated is equal to 
    \[
      e^{-x/16}\left(e^{-\frac{x}{16(d-1)}}\frac{ex}{4(d-1)}\right)^{d-1}\leq
      4^{d-1}e^{-x/16}.
    \]
    In particular, one has
    \[
      \|\mathds{1}_US_N\mathds{1}_V\|_{L^1\to L^{\infty}}\leq
      \frac{4}{\sqrt{2\pi d}}4^de^{-(N+1)D^2/16},
    \]
    hence the claim.
  \end{proof}
  Using the Schur test to estimate
  $\|\mathds{1}_US_N\mathds{1}_V\|_{L^2\to L^2}$ seems rather
  weak. Indeed, an easy bound is
  \[
    \|\mathds{1}_US_N\mathds{1}_V\|_{L^2\to L^2}\leq 1.
  \]
  Theorem \ref{thr:decay-op} beats this easy bound when $d\geq 3$
  under the condition
  \[
    D\geq 5\sqrt{\frac{d-1}{N+1}}.
  \]
  In particular, one has
  \begin{prop}\label{prop:op-dec}
    If $d\geq 3$, if $D\geq 10\sqrt{\frac{d}{N+1}}$ and if $U,V$ are
    two open sets of $(\S^2)^d$ at distance $D$, then
    \[
      \|\mathds{1}_US_N\mathds{1}_V\|_{L^2\to L^2}\leq
      \exp\left(-\frac{1}{21}(N+1)D^2\right).
        \]
      \end{prop}
      We will rely heavily on Proposition \ref{prop:op-dec} later
      on.

      Using the Schur test to estimate
      $\|\mathds{1}_US_N\mathds{1}_V\|_{L^2\to L^2}$ is very
      crude. We conclude this section with the following conjecture.
      
      \begin{conj}
        There exists a universal constant $c>0$ such that, for any
        integers $d,N$, for any open sets $U,V$ in $(\S^2)^d$, one has
        \[
          \|\mathds{1}_US_N\mathds{1}_V\|_{L^2\to L^2}\leq
          \exp(-cN\dist(U,V)^2).
          \]
        \end{conj}
        This conjecture is at least true if $U$ is a ball around one
        point, and $V$ is the complement of a larger ball around that
        same point. If we want to prove Theorem \ref{thr:conc-Linf} in
        the context of a large product of spheres, one would need to
        apply this conjecture to distances much shorter than $\sqrt{d}{N}$.

%%% Local Variables:
%%% mode: latex
%%% TeX-master: "Article.tex"
%%% End:

\section{Fractional decay of eigenfunctions without regularity}
\label{sec:fract-decay-eigenf}

In this section we prove Theorem \ref{thr:conc-Linf}.

  Let $f,u,\lambda$ be as above.

  Let us fix $U_0=\{f\geq \lambda+\epsilon\}$. Let (see also picture below)
  \begin{align*}
    a&=\epsilon^{\frac{2}{1+2\alpha}}N^{-\frac{\alpha}{2\alpha+1}}\\
    U_0'&=\left\{x\in U_0,\dist(x,\partial U_0)>a\right\}\\
    U_0''&=\left\{x\in U_0,\dist(x,\partial U_0)>2a\right\}\\
    V_0''&=M\setminus\left\{x\in U_0,\dist(x,\partial U_0)\geq 3a\right\}\\
    V_0'&=M\setminus\left\{x\in U_0,\dist(x,\partial U_0)\geq 4a\right\}\\
    V_0&=M\setminus\left\{x\in U_0,\dist(x,\partial U_0)\geq 5a\right\}.
  \end{align*}
  
  Note that $a=o(\epsilon)$ and $N^{-\frac 12}=o(a)$ as $N\to +\infty$. 
  We also let $\chi_0\in C^{\infty}(M,[0,1])$ be such that
  $\supp(\chi_0)\subset V_0''$ and $\supp(1-\chi_0)\subset U_0''$.
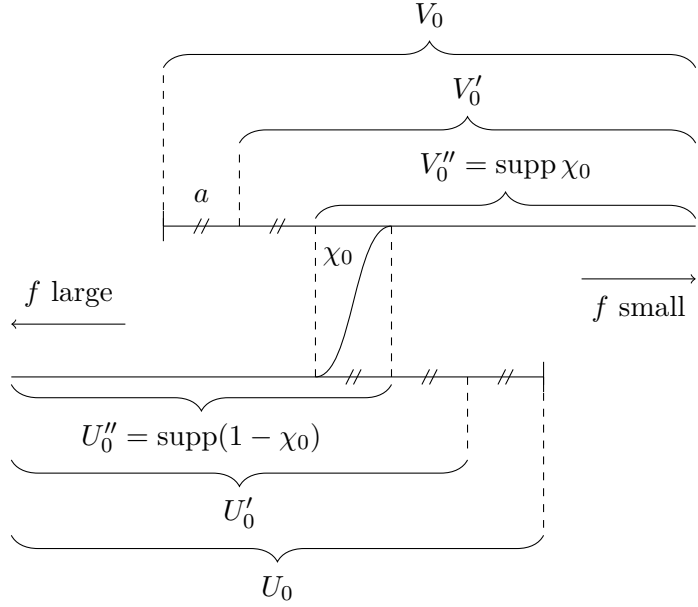
\begin{center}
  \begin{tikzpicture}
    \draw (0,0) -- (7,0);
    \draw (0,-0.2) -- (0,0.2);

    \draw(5,-2) -- (-2,-2);
    \draw(5,-2.2 ) -- (5,-1.8);
    %\draw (4.5,-3.7) node{$V_0$};
    \draw[dashed] (2,0) -- (2,-2);
    \draw[dashed] (3,0) -- (3,-2);
    \draw[dashed] (1,0) -- (1,1);
    \draw[dashed] (4,-2) -- (4,-3);
    \draw[dashed] (0,0) -- (0,2);
    \draw[dashed] (5,-2) -- (5,-4);
    
    \draw [decorate, decoration={brace, amplitude=10pt}] (2,0.1)--
    (7,0.1) node[midway, yshift=0.7cm]{$V_0''=\supp \chi_0$};
    \draw [decorate, decoration={brace, amplitude=10pt,raise=1cm}] (1,0.1)--
    (7,0.1) node[midway, yshift=1.7cm]{$V_0'$};
    \draw [decorate, decoration={brace, amplitude=10pt,raise=2cm}] (0,0.1)--
    (7,0.1) node[midway, yshift=2.7cm]{$V_0$};
    
    \draw [decorate, decoration={brace,mirror,amplitude=10pt}] (-2,-2.1)--
    (3,-2.1) node[midway, yshift=-0.7cm]{$U_0''=\supp (1-\chi_0)$};
    \draw [decorate, decoration={brace,mirror,amplitude=10pt,raise=1cm}] (-2,-2.1)--
    (4,-2.1) node[midway, yshift=-1.7cm]{$U_0'$};
    \draw [decorate, decoration={brace,mirror,amplitude=10pt,raise=2cm}] (-2,-2.1)--
    (5,-2.1) node[midway, yshift=-2.7cm]{$U_0$};

    \draw (2,-2) .. controls (2.5,-2) and (2.5,0) .. (3,0);

    \draw (2.3, -0.4) node{$\chi_0$};
    \draw (0.5,0.4) node{$a$};
    \draw (0.4,-0.1)--(0.5,0.1);
    \draw (0.5,-0.1)--(0.6,0.1);
    \draw (1.4,-0.1)--(1.5,0.1);
    \draw (1.5,-0.1)--(1.6,0.1);
    \draw (2.4,-2.1)--(2.5,-1.9);
    \draw (2.5,-2.1)--(2.6,-1.9);
    \draw (3.4,-2.1)--(3.5,-1.9);
    \draw (3.5,-2.1)--(3.6,-1.9);
    \draw (4.4,-2.1)--(4.5,-1.9);
    \draw (4.5,-2.1)--(4.6,-1.9);

    \draw [->](-0.5,-1.3)--(-2,-1.3) node[midway,yshift=0.4cm]{$f$
      large};
    \draw [->](5.5,-0.7)--(7,-0.7) node[midway,yshift=-0.4cm]{$f$ small};

  \end{tikzpicture}
\end{center}
Now
  \begin{align*}
    0&=\langle u,(f-\lambda)u\rangle\\
    &=\langle \chi_0 u,(f-\lambda)\chi_0u\rangle+2\langle
      (1-\chi_0)u,(f-\lambda)\chi_0u\rangle + \langle
      (1-\chi_0)u,(f-\lambda)(1-\chi_0)u\rangle.
  \end{align*}
  Since $f-\lambda\geq \epsilon$ on the support of $1-\chi_0$, one has
  \[
    \langle
      (1-\chi_0)u,(f-\lambda)(1-\chi_0)u\rangle\geq
      \epsilon\|u\|_{L^2(U_0'')}^2.
    \]

    Moreover,
    \[
      \langle (1-\chi_0)u,(f-\lambda)\chi_0u\rangle\leq
      \max_{U_0''\cap V_0''}(f-\lambda)\|u\|_{L^2(U_0''\cap V_0'')}^2.
    \]
    It remains to bound
    \[
      \langle \chi_0 u,(f-\lambda)\chi_0u\rangle=\langle
      (f-\lambda)u,\chi_0^2 u\rangle
      \] from above. To this end, observe that $S_N(f-\lambda)u=0$,
      and it remains to estimate
      \[
        \langle (f-\lambda)u,(1-S_N)\chi_0^2u\rangle=\iint_{x\in
          M,y\in
          M}(f(x)-\lambda)\overline{u(x)}S_N(x,y)\left[\chi_0^2(x)u(x)-\chi_0^2(y)u(y)\right]\dd
        y \dd x.
      \]
We first examine this integral restricted to $x\in U_0'\cap V_0'$,
that is,
\[
  A=\langle(f-\lambda)u,\mathds{1}_{U_0'\cap V_0'}[1-S_N]\chi_0^2u\rangle
\]
which we decompose as
\begin{align*}
  A&=A_1+A_2\\
  A_1&=\langle(f-\lambda)u,\mathds{1}_{U_0'\cap
       V_0'}[1-S_N]\mathds{1}_{M\setminus (U_0\cap V_0)}\chi_0^2u\rangle\\
  A_2&=\langle(f-\lambda)u,\mathds{1}_{U_0'\cap
       V_0'}[1-S_N]\mathds{1}_{U_0\cap V_0}\chi_0^2u\rangle.
\end{align*}
Since $\dist(U_0'\cap V_0',M\setminus (U_0\cap V_0))\geq a$, one can
apply Proposition \ref{prop:kernel-decay}, so that
\[
  |A_1|\leq
  C\max_{U_0'\cap V_0'}(f-\lambda)\exp(-c(Na^2)^{\alpha})\|u\|_{L^2(U_0'\cap V_0')}.
\]
Moreover, one has
\[
  |A_2|\leq \max_{U_0'\cap V_0'}(f-\lambda)\|u\|^2_{L^2(U_0\cap V_0)}.
  \]

Now we consider the integral restricted to $x\in M\setminus(U_0'\cap
V_0')$, that is
\[
  B=\langle \mathds{1}_{M\setminus(U_0'\cap
    V_0')}(f-\lambda)u,[1-S_N]\chi_0^2u\rangle.
\]
One has, since $\chi_0=0$ on $M\setminus V_0''$,
\begin{align*}
  B&=B_1+B_2+B_3\\
  B_1&=\langle \mathds{1}_{M\setminus(U_0'\cap
    V_0')}(f-\lambda)u,[1-S_N]\chi_0^2\mathds{1}_{U_0''\cap
    V_0''}\rangle\\
  B_2&=\langle \mathds{1}_{M\setminus
    V_0'}(f-\lambda)u,[1-S_N]\chi_0^2\mathds{1}_{M\setminus
    U_0''}\rangle\\
  B_3&=\langle \mathds{1}_{M\setminus
    U_0'}(f-\lambda)u,[1-S_N]\chi_0^2\mathds{1}_{M\setminus
       U_0''}\rangle.
\end{align*}
From Proposition \ref{prop:kernel-decay} one has
\begin{align*}
  |B_1|&\leq
  C\max_M(f-\lambda)\exp(-c(Na^2)^{\alpha})\|u\|_{L^2(U_0''\cap
    V_0'')}\\
  |B_2|&\leq C
  \max_M(f-\lambda)\exp(-c(4Na^2)^{\alpha})\|u\|_{L^2(M\setminus V_0')}.
\end{align*}
Moreover $\chi_0^2=1$ on $M\setminus U_0''$, so that, since $S_Nu=u$,
\begin{align*}
  |B_3|&=\langle \mathds{1}_{M\setminus
    U_0'}(f-\lambda)u,[1-S_N]\chi_0^2\mathds{1}_{M\setminus
    U_0''}\rangle\\
  &=-\langle \mathds{1}_{M\setminus
    U_0'}(f-\lambda)u,[1-S_N]\chi_0^2\mathds{1}_{U_0''}\rangle.
\end{align*}
Then again
\[
  |B_3|\leq C\max_{M\setminus
    U_0'}(f-\lambda)\exp(-c(N^2a)^{\alpha})\|u\|_{L^2(U_0'')}.
\]

To conclude, from
\[
  0=\langle \chi_0 u,(f-\lambda)\chi_0u\rangle+2\langle
      (1-\chi_0)u,(f-\lambda)\chi_0u\rangle + \langle
      (1-\chi_0)u,(f-\lambda)(1-\chi_0)u\rangle,
    \]
    we obtain the inequality
    \begin{multline*}
      c_0\|u\|_{L^2(V_0'')}^2\leq \max_{U_0''\cap
        V_0''}(f-\lambda)\|u\|_{L^2(U_0''\cap
        V_0'')}^2+C\max_{U_0'\cap
        V_0'}(f-\lambda)\exp(-c(Na^2)^{\alpha})\|u\|_{L^2(U_0'\cap
        V_0')}\\
      +\max_{U_0'\cap V_0'}(f-\lambda)\|u\|^2_{L^2(U_0\cap V_0)}+C\max_M(f-\lambda)\exp(-c(Na^2)^{\alpha})\|u\|_{L^2(U_0''\cap
        V_0'')}\\
      +C
  \max_M(f-\lambda)\exp(-c(4Na^2)^{\alpha})\|u\|_{L^2(M\setminus V_0')}+C\max_{M\setminus
    U_0'}\exp(-c(N^2a)^{\alpha})\|u\|_{L^2(U_0'')},
\end{multline*}
which we simplify into
\begin{equation*}
  \left(\epsilon-4C\max_M|f|e^{-c(N^2a)^{\alpha}}\right)\|u\|^2_{L^2(U_0'')}\leq
2C\max_M(f-\lambda)\|u\|_{L^2(U_0\cap V_0)}\left(\|u\|_{L^2(U_0\cap
    V_0)}+e^{-c(Na^2)^{\alpha}}\right).
\end{equation*}
Since $N^2a=N\epsilon^4\geq N^{\delta}$, let us restrict ourselves to
$N$ large enough (depending on $\delta$) so that
\[
  4C\max_M|f|e^{-c(N^2a)^{\alpha}}\leq \epsilon/2.
  \]

In conclusion, one has the following dichotomy.
\begin{itemize}
\item Either $\|u\|_{L^2(U_0\cap
    V_0)}\leq e^{-c(Na^2)^{\alpha}}$, in which case
  \[
    \|u\|^2_{L^2(U_0'')}\leq 4\frac{C}{\epsilon}e^{-c(Na^2)^{\alpha}}.
  \]
\item Or $\|u\|_{L^2(U_0\cap
    V_0)}\geq e^{-c(Na^2)^{\alpha}}$, so that
  \[
    \|u\|^2_{L^2(U_0'')}\leq 4\frac{C}{\epsilon}\|u\|^2_{L^2(U_0\cap V_0)}.
  \]
\end{itemize}
In the second case, one proceeds to an induction, letting
  \[
    U_1=\mathop{int}(M\setminus V_0)
  \]
  where $\mathop{int}(E)$ is the interior of the set $E$. One has then
  \begin{align*}
    &U_1=\{x\in M, \dist(x,U_0)>5a\}\\
    &\|u\|^2_{L^2(U_1)}\leq \cfrac{4\cfrac{C}{\epsilon}}{1+4\cfrac{C}{\epsilon}}\|u\|^2_{L^2(U_0)}.
  \end{align*}
  We proceed in the induction, considering sets $U_k,V_k$, and so
      on, until one of these conditions is satisfied: $k=\frac{\epsilon}{6a}$ or
      \[
        \|u\|^2_{L^2(U_k)}\leq 4\frac{C}{\epsilon}e^{-c(Na^2)^{\alpha}}.
      \]
      If we have reached $k=\frac{\epsilon}{6a}$, then $U_k$ is the
      set of points at distance at least $\frac{5}{6}\epsilon+O(a)$ of $U_0$, and
      \[
        \|u\|^2_{L^2(U_k)}\leq
        \left[\cfrac{4\cfrac{C}{\epsilon}}{1+4\cfrac{C}{\epsilon}}\right]^{k}\leq
        \exp(-c\epsilon k).
      \]
      In the other case, the last iteration $U_k$ contains the set of
      points at distance $\epsilon$ of $U_0$, and is such that
      \[
        \|u\|^2_{L^2(U_k)}\leq 4\frac{C}{\epsilon}e^{-c(Na^2)^{\alpha}}.
      \]
      Now since $k=\frac{\epsilon}{5a}$, one has
      \[
        \epsilon k\approx c(Na^2)^{\alpha}\approx N^{\frac{\alpha}{2\alpha+1}},
      \]
      where $\approx$ means ``up to some constant''.
      
      This concludes the proof.

\section{Decay of eigenfunctions for Lipschitz symbols}
\label{sec:decay-eigenf-lipsch}

In this section we prove Theorems \ref{thr:conc-Lip} and
\ref{thr:conc-C2-low}. They respectively follow from the two
following weighted estimates:

\begin{prop}\label{prop:Agmon-fixed-M}
  Let $M$ be a $C^{1,1}$ K\"ahler manifold. Let $\rho\in
  \mathrm{Lip}(M,\R)$.

  There exist two constants $c>0$ and $C>0$ such that, for every
  $\alpha\in \R$ with $|\alpha|<c$, for every
  $f\in \mathrm{Lip}(M,\R)$ with Lipschitz constant $K$, for any $N\in \N$, for any $u\in
  H^0(M,L^{\otimes N})$ such that $T_N(f)u=\lambda u$ for some
  $\lambda\in \R$, one has
  \[
    \int_Me^{2\alpha\sqrt{N}\rho(x)}\left(f(x)-\lambda-CK|\alpha|N^{-\frac
        12}\right)|u(x)|^2\dd x \leq 0.
  \]
  Moreover, the constants $c$ and $C$ only depend on the Lipschitz
  constant of $\rho$.
\end{prop}

\begin{prop}
  \label{prop:Agmon-induction}
  Let $M$ be a $C^{1,1}$ K\"ahler manifold, and let $f\in C^{1,1}(M,\R)$
  with $\min(f)=0$. Let $u\in H^0(M,L^{\otimes N})$ be such that
  $T_N(f)u=\lambda u$, with $\lambda=O(N^{-1})$. For $k\in \N$ and $\epsilon>0$, define
  \[
    g_k^{\epsilon}=\begin{cases}f-\lambda&\text{ if }k=0\\
      \max(f-\lambda,N^{-1+\frac{1}{2^k}+2\epsilon})&\text{
        otherwise.}
    \end{cases}
  \]
  If, for some $k\geq 0$, for all $\epsilon>0$, there exists $C_k>0$ and $c_k>0$ such that,
  for all $|\alpha|<c_k$, one has
  \[
    \int_Me^{2\alpha\sqrt{N}\sqrt{f(x)}}\left(g^{\epsilon}_k(x)-C_kN^{-1+\frac{1}{2^{k+1}}+\epsilon}\right)|u(x)|^2\dd
    x \leq 0,
  \]
  then for all $\epsilon>0$ there exists $C_{k+1}>0$ and $c_{k+1}>0$ such that, for all
  $|\alpha|<c_{k+1}$, one has
  \[
    \int_Me^{2\alpha\sqrt{N}\sqrt{f(x)}}\left(g^{\epsilon}_{k+1}(x)-C_{k+1}N^{-1+\frac{1}{2^{k+2}}+\epsilon}\right)|u(x)|^2\dd
    x \leq 0.
    \]
  \end{prop}

We postpone the proof of these estimates, and
first use them to prove Theorems \ref{thr:conc-Lip} and \ref{thr:conc-C2-low}.

{\bfseries Proof of Theorem \ref{thr:conc-Lip}.}

Letting $M,f,u,\lambda$ be as in Proposition \ref{prop:Agmon-fixed-M},
we choose $c,C$ corresponding to the Lipschitz constant 1; indeed we
will choose $\rho=\dist(\cdot,U)$ where $U$ will be defined
later.

Now, for every $|\alpha|<c$, one has, by Proposition \ref{prop:Agmon-fixed-M}
\[
  0\geq \int_Me^{2\alpha\sqrt{N}\rho(x)}\left(f(x)-\lambda-CK|\alpha|N^{-\frac
      12}\right)|u(x)|^2\dd x.\]
Let us decompose this integral in two pieces, corresponding to the
sign of $f-\lambda-CK|\alpha|N^{-\frac 12}$: with
$\lambda_1=\lambda+CK|\alpha|N^{-\frac 12}$ and
$\lambda_2=\lambda+2CK|\alpha|N^{-\frac 12}$, one has
\begin{align*}
  0\geq &\int_{\{f\geq \lambda_2\}}e^{2\alpha\sqrt{N}\rho(x)}\left(f(x)-\lambda_1\right)|u(x)|^2\dd x\\
  &+\int_{\{\lambda_1\leq f \leq \lambda_2\}}e^{2\alpha\sqrt{N}\rho(x)}\left(f(x)-\lambda_1\right)|u(x)|^2\dd x\\
  &+\int_{f\leq
  \lambda_1}e^{2\alpha\sqrt{N}\rho(x)}\left(f(x)-\lambda_1\right)|u(x)|^2\dd
  x.
\end{align*}
The second contribution is positive, and one can remove it; with
$\rho=\dist(\cdot,\{f\leq \lambda_1\})$, this yields
\[
  CK|\alpha|N^{-\frac 12}\int_{\{f\geq
    \lambda_2\}}e^{2\alpha\sqrt{N}\dist(x,\{f\leq
    \lambda_1\})}|u(x)|^2\dd x\leq (\lambda_1-\min(f))\int_{\{f\leq
    \lambda_1\}}|u(x)|^2\dd x.
  \]
    To conclude the proof of Theorem \ref{thr:conc-Lip}, we let
    $\alpha=\frac c2$; then for $\epsilon>CN^{-\frac 12}$, with
    \[
      W=\{x\in M, \dist(x,\{f\geq \lambda_2\})>\epsilon\},
    \]
    on $W$ one has $\dist(\cdot,\{f\leq
    \lambda_1\})>\epsilon$, so that
    \[
      \int_W|u|^2\leq e^{-c\epsilon\sqrt{N}}\int_{\{f\geq
    \lambda_2\}}e^{2\alpha\sqrt{N}\dist(x,\{f\leq
    \lambda_1\})}|u(x)|^2\dd x\leq CN^{\frac
    12}e^{-c\epsilon\sqrt{N}}\|u\|_{L^2}^2.
\]
This concludes the proof.

{\bfseries Proof of Theorem \ref{thr:conc-C2-low}.}
Let $f\in C^{1,1}(M,\R)$ with $\min(f)=0$. It is well-known that
$\sqrt{f}$ is Lipschitz-continuous. In particular, the initialisation
of the induction in Proposition \ref{prop:Agmon-induction} is given by
Proposition \ref{prop:Agmon-fixed-M}, and thus, for all $k\in \N$, for
all $\epsilon>0$, one
has
\[
  \int_Me^{2\alpha\sqrt{N}\sqrt{f(x)}}\left(g^{\epsilon}_k(x)-C_kN^{-1+\frac{1}{2^{k+1}}+\epsilon}\right)|u(x)|^2\dd
  x \leq 0,
\]
for $|\alpha_k|<c_k(\epsilon)$.

Let $\delta>0$; for some $k$ large enough and for some $\epsilon>0$
one has $\delta=\frac{1}{2^{k+1}}+\epsilon$.

We now proceed as in the proof of Theorem \ref{thr:conc-Lip}: let
$\lambda_1=\lambda+C_kN^{-1+\delta}$ and
$\lambda_2=\lambda+2C_kN^{-1+\delta}$. Then
\[
  0\geq \int_{\{f\geq
    \lambda_2\}}e^{2\alpha\sqrt{N}\sqrt{f}}(g_k^{\epsilon}-\lambda_1+\lambda)|u(x)|^2\dd
  x + \int_{\{f\leq
    \lambda_1\}}e^{2\alpha\sqrt{N}\sqrt{f}}(g_k^{\epsilon}-\lambda_1+\lambda)|u(x)|^2
  \dd x.
  \]
  In particular,
  \[
    C_kN^{-1+\delta}e^{2\alpha\sqrt{2C_k}N^{\frac{\delta}{2}}}\int_{\{f\geq
      \lambda_2\}}|u|^2\leq
    C_kN^{-1+\delta}e^{2\alpha\sqrt{C_k}N^{\frac{\delta}{2}}}\int_{\{f\leq
        \lambda_1\}}|u|^2,
    \]
    so that, finally,
    \[
      \int_{\{f\geq
      \lambda+C_kN^{-1+\delta}\}}|u|^2\leq
    e^{-2(\sqrt{2}-1)\alpha\sqrt{C_k}N^{\frac{\delta}{2}}}.
    \]

    The proof of Propositions \ref{prop:Agmon-fixed-M} and
    \ref{prop:Agmon-induction} rely on the following commutator estimates.
\begin{lem}\label{lem:commuts-fixed-M}
  Let $M$ be a $C^{1,1}$ K\"ahler manifold. Let $\rho\in
  \mathrm{Lip}(M,\R)$.

  There exist two constants $c>0$ and $C>0$ such that, for any
  $\alpha\in \R$ with $|\alpha|<c$, for any $f\in
  \mathrm{Lip}(M,\R)$, if $K$ denotes the Lipschitz constant of $f$,
  one has
  \begin{align*}
    \|\exp(-\alpha\sqrt{N}\rho)[S_N,\exp(2\alpha\sqrt{N}\rho)]\exp(-\alpha\sqrt{N}\rho)\|_{L^2\to
    L^2}&\leq C|\alpha|\\
    \|\exp(\alpha\sqrt{N}\rho)[f,S_N]\exp(-\alpha\sqrt{N}\rho)\|_{L^2\to L^2}&\leq CKN^{-\frac 12}.
  \end{align*}
  Moreover the constants $c,C$ depend only on the Lipschitz constant
  of $\rho$.
\end{lem}
\begin{proof}
  We first prove the second bound; the first bound is a consequence of
  the second one.
  
  Recall from Proposition \ref{prop:pointwise-decay} that the kernel of
  $S_N$ is bounded everywhere: there exists $C_0>0,c_0>0$ such that
  for all $(x,y)\in M\times M$, for all $N\in \N$, one has
  \[
    |S_N(x,y)|\leq CN^d\exp(-c\sqrt{N}\dist(x,y)).
  \]
  Here $d$ denotes again the dimension of $M$.

  The kernel of $\exp(\alpha\sqrt{N}\rho)[f,S_N]\exp(-\alpha\sqrt{N}\rho)$ is
  \[
    (x,y)\mapsto
    S_N(x,y)(f(x)-f(y))\exp(\alpha\sqrt{N}(\rho(x)-\rho(y))).
  \]
  Let $L$ denote the Lipschitz contant of $\rho$; then the kernel
  above is everywhere bounded by
  \[
    (x,y)\mapsto C_0K\dist(x,y)N^d\exp((-c_0+\alpha L)\dist(x,y)\sqrt{N}).
  \]
  Let $c=\frac{c_0}{L}$. For $|\alpha|<c$, the Schur norm of this
  kernel is smaller than
  \[
    C_0K\sup_{y\in M}\int_{x\in
      M}\dist(x,y)\exp\left(-\frac{c_0}{2}\dist(x,y)\sqrt{N}\right)\leq
    C_1KN^{-\frac 12}.
  \]

  For the first bound, we proceed by differentiation with respect to
  $\alpha$. The statement clearly holds for $\alpha=0$, in which case
  $[S_N,1]=0$. With
  \[
    T(a)=\exp(-\alpha\sqrt{N}\rho)[S_N,\exp(2\alpha\sqrt{N}\rho)]\exp(-\alpha\sqrt{N}\rho),
  \]
  one has
  \[
    T'(a)=\sqrt{N}\left[\exp(-\alpha\sqrt{N}\rho)[S_N,\rho]\exp(\alpha\sqrt{N}\rho)-\exp(\alpha\sqrt{N}\rho)[S_N,\rho]\exp(-\alpha\sqrt{N}\rho)\right].
  \]
  
  We can now apply the second bound (with $f=\rho$); as long as
  $|\alpha|<c$, one has
  \[
    \|T'(a)\|_{L^2\to L^2}\leq 2C_1L\sqrt{N}.
  \]
  This concludes the proof.
\end{proof}

{\bfseries Proof of Proposition
\ref{prop:Agmon-fixed-M}.}

Without loss of generality, one can assume $\lambda=0$ by replacing
$f$ with $f-\lambda$. As in
\cite{kordyukov_semiclassical_2018}, since $S_Nfu=0$, one can write
\begin{align*}
      \langle
      e^{\alpha\sqrt{N}\rho}fu,e^{\alpha\sqrt{N}\rho}u\rangle&=\langle
                                                     [S_N,e^{2\alpha\sqrt{N}\rho}]fu,u\rangle\\
      &=\langle [S_N,e^{2\alpha\sqrt{N}\rho}][f,S_N]u,u\rangle.
\end{align*}
To use Lemma \ref{lem:commuts-fixed-M}, we
need to introduce a few supplementary exponential factors:
\[
  \langle [S_N,e^{2\alpha\sqrt{N}\rho}][f,S_N]u,u\rangle
  =\langle
  e^{-\alpha\sqrt{N}\rho}[S_N,e^{2\alpha\sqrt{N}\rho}]e^{-\alpha\sqrt{N}\rho}e^{\alpha\sqrt{N}\rho}[f,S_N]e^{-\alpha\sqrt{N}\rho}e^{2\alpha\sqrt{N}\rho}u\rangle.
\]
Hence, if $K$ denotes the Lipschitz constant of $f$ one has, by Lemma \ref{lem:commuts-fixed-M}
\[
  \langle [S_N,e^{2\alpha\sqrt{N}\rho}][f,S_N]u,u\rangle
  \leq C^2|\alpha|KN^{-\frac 12}\|e^{\alpha\sqrt{N}\rho}u\|_{L^2}^2.
\]
This concludes the proof.

{\bfseries Proof of Proposition \ref{prop:Agmon-induction}.}

Let us modify the proof of Proposition \ref{prop:Agmon-fixed-M} in this
context where $u$ is an eigenfunction of $T_N(g^{\epsilon}_{k+1})$ only up to
some error, given by the induction hypothesis.

Let $\rho$ be a
Lipschitz function. Then, for all $\alpha'$, one has
\[
  \langle
  e^{\alpha'\sqrt{N}\rho}g_{k+1}u,e^{\alpha'\sqrt{N}\rho}u\rangle-\langle
  [S_N,e^{2\alpha'\sqrt{N}\rho}][g^{\epsilon}_{k+1},S_N]u,u\rangle=-\langle e^{\alpha'
    \sqrt{N}\rho}S_N(g_{k+1}u),e^{\alpha'\sqrt{N}\rho}u\rangle.
\]
If $f\in C^{1,1}$ is non-negative then $\sqrt{f}$ is Lipschitz-continuous;
we now fix $\rho=\sqrt{f}$. 
For $|\alpha'|$ small enough, we want to estimate
\[
  e^{2\alpha'\sqrt{N}\sqrt{f}}S_N(g^{\epsilon}_{k+1}u)=\left(e^{2\alpha'\sqrt{N}\sqrt{f}}S_Ne^{-2\alpha'\sqrt{N}\sqrt{f}}\right)e^{2\alpha'\sqrt{N}\sqrt{f}}(g_{k+1}^{\epsilon}-f+\lambda)u.
\]
The operator
$e^{2\alpha'\sqrt{N}\sqrt{f}}S_Ne^{-2\alpha'\sqrt{N}\sqrt{f}}$ is
bounded independently of $N$ from $L^2$ to $L^2$ if $|\alpha'|$ is small
enough. Moreover,
$g_{k+1}^{\epsilon}-f+\lambda$ is supported on $\{f\geq N^{-1+\frac{1}{2^{k+1}}+2\epsilon}\}$, so
that
\[
  \int_M
  e^{4\alpha'\sqrt{N}\sqrt{f}}|g_{k+1}^{\epsilon}-f+\lambda|^2|u|^2\leq C\int_{\{f\geq
    \lambda+N^{-1+\frac{1}{2^{k+1}}+2\epsilon}\}}e^{4\alpha'\sqrt{N}\sqrt{f}}|u|^2.\]
Let $\alpha>\frac{c_k(\epsilon)}{2}$ (so that the weighted estimate of the
induction is satisfied). Then, on $\{f\geq
    \lambda+N^{-1+\frac{1}{2^{k+1}}+2\epsilon}\}$, one has, for
    $|\alpha'|$ small enough, for some $c>0$,
\[
  e^{4\alpha'\sqrt{N}\sqrt{f}}\leq
  e^{-cN^{\frac{1}{2^{k+2}}+\epsilon}}e^{2\alpha\sqrt{N}\sqrt{f}},\]
so that
\[
  \int_{\{f\geq
    \lambda+N^{-1+\frac{1}{2^{k+1}}+2\epsilon}\}}e^{4\alpha'\sqrt{N}\sqrt{f}}|u|^2\leq
  e^{-cN^{\frac{1}{2^{k+2}}+\epsilon}}\int_{\{f\geq \lambda+N^{-1+\frac{1}{2^{k+1}}+2\epsilon}\}}e^{2\alpha\sqrt{N}\sqrt{f}}|u|^2.
\]
By hypothesis, one has
\[
\int_Me^{2\alpha\sqrt{N}\sqrt{f}(x)}\left(g_k^{\epsilon}(x)-C_kN^{-1+\frac{1}{2^{k+1}}+\epsilon}\right)|u(x)|^2\dd x\leq 0.
\]
In particular,
\begin{multline*}
  0\geq CN^{-1+\frac{1}{2^{k+1}}+\epsilon}\int_{\{f\geq \lambda+2CN^{-1+\frac{1}{2^{k+1}}+\epsilon}\}}e^{2\alpha\sqrt{N}\sqrt{f}}|u|^2\\-C|\alpha|N^{-1+\frac{1}{2^{k+1}}+\epsilon}e^{C|\alpha|N^{\frac{1}{2^{k+2}}+\frac{\epsilon}{2}}}\int_{\{f\leq \lambda+N^{-1+\frac{1}{2^{k+1}}+\epsilon}\}}|u|^2,
\end{multline*}
so that
\[
  \int_{\{f\geq \lambda+2CN^{-1+\frac{1}{2^{k+1}}+\epsilon}\}}e^{2\alpha\sqrt{N}\sqrt{f}}|u|^2\leq Ce^{C|\alpha|N^{\frac{1}{2^{k+2}}+\frac{\epsilon}{2}}}.
\]
Hence, for some $c'>0$, one has
\[
  \|e^{2\alpha'\sqrt{N}\sqrt{f}}S_N(gu)\|_{L^2}^2\leq Ce^{-cN^{\frac{1}{2^{k+2}}+\epsilon}}e^{C|\alpha|N^{\frac{1}{2^{k+2}}+\frac{\epsilon}{2}}}\leq Ce^{-c'N^{\frac{1}{2^{k+2}}+\epsilon}},
  \]
  so that
  \[
    \left|\langle
  e^{\alpha'\sqrt{N}\rho}gu,e^{\alpha'\sqrt{N}\rho}u\rangle-\langle
  [S_N,e^{2\alpha'\sqrt{N}\rho}][g,S_N]u,u\rangle\right|\leq
Ce^{-c'N^{\frac{1}{2^{k+2}}+\epsilon}}.
\]
We can now, up to this error, reproduce the end of the proof of
Proposition \ref{prop:Agmon-fixed-M}. Since the Lipschitz constant of
$g_k^{\epsilon}$ is $N^{-\frac 12+\frac{1}{2^{k+2}}+\epsilon}$, this yields
\[
  \int_Me^{2\alpha'\sqrt{N}\sqrt{f}(x)}\left(g(x)-\lambda-C|\alpha|N^{-1+\frac{1}{2^{k+2}}\epsilon}\right)|u(x)|^2\leq 0.
\]

This concludes the proof.

\section{Weighted estimates: uniformity in the dimension}
\label{sec:weighted-estimates}

Kordyukov \cite{kordyukov_semiclassical_2018} has proposed a method for obtaining weighted
estimates for eigenfunctions of Toeplitz operators, based on the
ellipticity of the Hodge Laplacian (thus generalizing results on the
off-diagonal decay of the Szeg\H{o} projector).

In this section we revisit the proof of Theorem 1.3 in \cite{kordyukov_semiclassical_2018}, while making the
dependency on the geometry more explicit.

Let $M$ be a quantizable Kähler manifold of complex dimension $d$, with $L$ its prequantum
bundle. If $\nabla^N$ is the Levi-Civita holomorphic connection on
$L^{\otimes N}$, then $H^0(M,L^{\otimes N})$ is the kernel of
\[\Box_N=(\nabla^N)^*\nabla^N-\pi\dim(M)N.\]

Let $\rho\in C^2(M,\R)$ and $\alpha\in \R$. Conjugating $\Box_N$
with $e^{\alpha \sqrt{N}\rho}$ yields
\[
  \Box_{N;\alpha}=\exp(\alpha\sqrt{N}\rho)\Box_N\exp(-\alpha\sqrt{N}\rho)=\Box_N+\alpha
  A_N+\alpha^2B_N,
\]
where, given a local orthonormal frame $\{e_j\}_{1\leq j\leq 2d}$ of
$TX$,
\begin{align}\label{eq:ctrl-def-conn}
  A_N&=\sqrt{N}\sum_{j=1}^{2d}\left[\nabla_{e_j}^{N}\circ \dd
    \rho(e_j)+d\rho(e_j)\circ \nabla_{e_j}^N+\dd
       \rho\left(\nabla^{TM}_{e_j}e_j\right)\right]=\sqrt{N}(\Delta
       \rho+2\nabla \rho \cdot
  \nabla^N)\\
  B_N&=-N\|\nabla \rho\|^2
\end{align}
Here, $\nabla$
is the Riemannian gradient.

In this section, we consider an integrable K\"ahler manifold of the
form $M=M_0^{d'}$, and obtain estimates with explicit dependence on
$d'$. Throughout the section, the constants appearing are, unless
otherwise noted, independent on $d'$.

If $M$ is a product of manifolds $M=(M_0)^{d'}$, then there holds a uniform bound on the
spectral gap of $\Box_N$.

\begin{prop}\label{prop:bound-gap-Hodge}
  Let $M_0$ be a compact, quantizable K\"ahler manifold of regularity
  $C^{1,1}$. There exists $C_0>0$, $\mu>0$ such that the following is true.

  Let $d'\in \N$ and let $M=M_0^{d'}$. For $N\in \N$, we let $\Box_N$
  be the Hodge Laplacian over $M$ with semiclassical parameter
  $\frac{1}{N}$. Then for any $\lambda\in \C$ such that $|\lambda|=\mu_0$,
  one has
  \[
    \left\|\left(\lambda-\frac{1}{N}\Box_N\right)^{-1}\right\|^2_{L^2(M)\to
      L^2(M,L^{\otimes N})}+\frac{1}{\sqrt{N}}\left\|\left(\lambda-\frac{1}{N}\Box_N\right)^{-1}\right\|^2_{L^2(M)\to
      \dot{H}_1(M,L^{\otimes N})}\leq C_0,
  \]
  where the $\dot{H}_1$ quasinorm on sections of $L^{\otimes N}$ is
  defined as
  \[
    \|u\|^2_{\dot{H}_1(M,L^{\otimes
        N})}=\int_M\|\nabla^{N}u(x)\|^2_{\ell^2(TX\otimes
      L^{\otimes N})}\dd
    \mathop{Vol}(x).
  \]
\end{prop}
\begin{proof}
  The claim is true for $d'=1$, where it follows from the
  usual Hörmander-Kohn estimate \cite{hormander_l2_1965}. Indeed, in
  this case $\Box_{N,M_0}$ is a
  self-adjoint operator on $L^2(M_0,L_0^{\otimes N})$ and
  this estimate implies that \[\sigma(\Box_{N,M_0})\subset
  \{0\}\cap [CN,+\infty)\] for some $C>0$.

  If $(u_j)_{j\in \N}$ is an orthonormal basis of
  eigenfunctions of $\Box_{N,M_0}$, with eigenvalues
  $(\mu_j)_{j\in \N}$, then the eigenfunctions
  of $\Box_{N}$ are tensor products of the $u_j$'s (acting
  on different variables), since
  \[
    \Box_{N}=\sum_{j=1}^{d'}I^{\otimes j-1}\otimes
    \Box_{N,M_0}\otimes I^{d'-j};
  \]
  moreover the eigenvalues of $\Box_{N}$ are the sums of
  $d'$ eigenvalues of $\Box_{N,M_0}$.
  In particular, the spectral
  gap on $\Box_{N,M_0}$ propagates to $\Box_{N}$, leading
  to
  \[
    \left\|\left(\lambda-\frac{1}{N}\Box_N\right)^{-1}\right\|^2_{L^2(M)\to
      L^2(M,L^{\otimes N})}\leq \frac{C_0}{2}.
  \]
  for $|\lambda|=\frac{1}{2C}$ and $C_0=\frac{4}{C}$.
  
  Moreover, the family $(u_j)_{j\in \N}$ is also orthogonal
  for the $\dot{H}_1$ product, since
  \[
    \langle u_j,u_k\rangle_{\dot{H}_1}=\langle
    \nabla^{N}u_j,\nabla^Nu_k\rangle_{L^2}=\mu_k\langle
    u_j,u_k\rangle_{L^2}.
  \]
  Thus the estimate on the operator norm $L^2\to \dot{H}_1$
  also propagates from $M_0$ to $M$, which concludes the proof.
\end{proof}
By the usual resolvent identity, this leads to a spectral gap
on $\Box_{N; \alpha}$ for $|\alpha|$ small.
\begin{prop}\label{prop:bound-gap-Hodge-def}
  In the situation of Proposition \ref{prop:bound-gap-Hodge}, let
  $\rho\in \mathrm{Lip}(M,\R)$. For all $\alpha$ such that
  \[
   |\alpha|\leq \min\left[\|\nabla
      \rho\|_{L^{\infty}}^{-1},\frac 1{2C_0} \left(N^{-\frac 12}\|\Delta
        \rho\|_{L^{\infty}}+3\|\nabla \rho\|_{L^{\infty}}\right)^{-1}\right],
  \]
  one has
  \[
    \left\|\left(\lambda-\frac{1}{N}\Box_{N;\alpha}\right)^{-1}\right\|^2_{L^2(M)\to
      L^2(M,L^{\otimes N})}+\frac{1}{\sqrt{N}}\left\|\left(\lambda-\frac{1}{N}\Box_{N;\alpha}\right)^{-1}\right\|^2_{L^2(M)\to
      \dot{H}_1(M,L^{\otimes N})}\leq 2C_0.\]
\end{prop}
\begin{proof}
  One has
  \[
    \left(\lambda-\frac{1}{N}\Box_{N;
        \alpha}\right)^{-1}-\left(\lambda-\frac{1}{N}\Box_N\right)^{-1}=\frac
    1N
    \left(\lambda-\frac{1}{N}\Box_{N;\alpha}\right)^{-1}\left(\alpha
      A_N+\alpha^2B_N\right)\left(\lambda-\frac{1}{N}\Box_N\right)^{-1}.
  \]
  Here $A_N$ and $B_N$ are given by \eqref{eq:ctrl-def-conn}. Writing
  $A_N=A_{N,0}+A_{N,1}\cdot \nabla^N$ where $A_{N,0},A_{N,1}$
  are respectively $\Delta \rho$ and $2\nabla \rho$,  one has
  \[
    \frac 1N\left\|\alpha A_{N,0}+\alpha^2B_N\right\|_{L^2\to
      L^2}\leq |\alpha| N^{-\frac 12}\|\Delta
      \rho\|_{L^{\infty}}+\alpha^2\|\nabla \rho\|^2_{L^{\infty}}
  \]
  and
  \[
    \frac 1N \|A_{N,1}\cdot \nabla^N\|_{\dot{H}_1\to L^2}\leq
    2|\alpha| N^{-\frac 12}\|\nabla \rho\|_{L^{\infty}}.
  \]
  In particular, by Proposition \ref{prop:bound-gap-Hodge},
  \[
    \left\|\left(\alpha A_N+\alpha^2
        B_N\right)\left(\lambda-\frac{1}{N}\Box_N\right)^{-1}\right\|_{L^2\to
      L^2}\leq C_0|\alpha|\left(N^{-\frac 12}\|\Delta
      \rho\|_{L^{\infty}}+\|\nabla
      \rho\|_{L^{\infty}}\left(2+|\alpha|
        \|\nabla \rho\|_{L^{\infty}}\right)\right)
  \]
  so that, if
  \[
    |\alpha|\leq \min\left[\|\nabla
      \rho\|_{L^{\infty}}^{-1},\frac 1{2C_0} \left(N^{-\frac 12}\|\Delta
        \rho\|_{L^{\infty}}+3\|\nabla \rho\|_{L^{\infty}}\right)^{-1}\right]
  \]
  then
  \[
    2+|\alpha|
      \|\nabla \rho\|_{L^{\infty}}\leq 3.
  \]
  In particular,
  \[
    \left\|\left(\alpha A_N+\alpha^2
        B_N\right)\left(\lambda-\frac{1}{N}\Box_N\right)^{-1}\right\|_{L^2\to
      L^2}\leq \frac 12.
  \]
  Hence, the operator $I-\left(\alpha A_N+\alpha^2
    B_N\right)\left(\lambda-\frac{1}{N}\Box_N\right)^{-1}$
  is invertible on $L^2$, with operator norm bounded by
  $2$, so that the resolvent identity yields
  \begin{align*}
    \left\|\left(\lambda-\frac{1}{N}\Box_{N;
    \alpha}\right)^{-1}\right\|_{L^2\to L^2}&\leq 2
                                              \left\|\left(\lambda-\frac{1}{N}\Box_N\right)^{-1}\right\|_{L^2\to
                                              L^2}\\ \left\|\left(\lambda-\frac{1}{N}\Box_{N;
    \alpha}\right)^{-1}\right\|_{L^2\to \dot{H}_1}&\leq 2 \left\|\left(\lambda-\frac{1}{N}\Box_N\right)^{-1}\right\|_{L^2\to \dot{H}_1}.          
  \end{align*}
  One can then conclude from Proposition \ref{prop:bound-gap-Hodge}.
\end{proof}
\begin{rem}
  Proposition \ref{prop:bound-gap-Hodge-def} can be used to obtain
  off-diagonal exponential estimates for the kernel of
  the Szeg\H{o} projector. For fixed $d'$ and $\rho$, $|\alpha|$ is
  bounded by a constant, which limits this method to
  a decay of the form $\exp(-\sqrt{N}\dist(x,y))$.

  As $d'$ increases, using a similar construction as in Subsection
  \ref{sec:construction-weight}, this method is able to yield, at best, a decay of the
  form
  \[
    \|\mathds{1}_US_N\mathds{1}_V\|\leq C\exp(-c_1Nd^{-\frac 12}\dist(U,V)),
  \]
  which is too weak for our purpose; in particular, the
  more elementary estimate of Proposition \ref{thr:decay-op} beats
  this estimate on most of $M\times M$.
\end{rem}
Following \cite{kordyukov_semiclassical_2018} we then obtain a
dimension-independent version of Lemma \ref{lem:commuts-fixed-M}.
\begin{lem}\label{lem:commut}
  In the situation of Proposition \ref{prop:bound-gap-Hodge-def},
  there exists $C_1(M_0)$ such that
  \[
    \|\exp(-\alpha\sqrt{N}\rho)[S_N,\exp(2\alpha\sqrt{N}\rho)]\exp(-\alpha\sqrt{N}\rho)\|_{L^2\to
      L^2}\leq C_1|\alpha|\left[N^{-\frac 12}\|\Delta
      \rho\|_{L^{\infty}}+\|\nabla
      \rho\|_{L^{\infty}}\right].\]
  Moreover, for every $f\in C^2(M,\R)$, one has
  \[
    \|\exp(\alpha\sqrt{N}\rho)[f,S_N]\exp(-\alpha\sqrt{N}\rho)\|_{L^2\to L^2}\leq C_1\left[N^{-1}\|\Delta
        f\|_{L^{\infty}}+N^{-\frac 12}\|\nabla
        f\|_{L^{\infty}}\left(1+\|\nabla \rho\|_{L^{\infty}}\right)\right].
    \]
  \end{lem}
  \begin{proof}
    By Proposition \ref{prop:bound-gap-Hodge} and the spectral gap
    property, the Szeg\H{o} kernel is
    given by the following integral:
    \[
      S_N=\frac{1}{2i\pi}\oint_{|\lambda|=\mu_0}\left(\lambda-\frac{1}{N}\Box_N\right)^{-1}.
    \]
    In particular, one has
    \begin{multline*}
      \exp(-\alpha\sqrt{N}\rho)S_N\exp(\alpha\sqrt{N}\rho)-\exp(\alpha\sqrt{N}\rho)S_N\exp(-\alpha\sqrt{N}\rho)\\=\frac{1}{2i\pi}\oint_{|\lambda|=\mu_0}\left[\left(\lambda-\frac{1}{N}\Delta_{N;\alpha}\right)^{-1}-\left(\lambda-\frac{1}{N}\Delta_{N;-\alpha}\right)^{-1}\right]\\
      =\frac{2\alpha}{N} \frac{1}{2i\pi}\oint_{|\lambda|=\mu_0}\left(\lambda-\frac{1}{N}\Delta_{N;\alpha}\right)^{-1}A_N\left(\lambda-\frac{1}{N}\Delta_{N;\alpha}\right)^{-1}.
    \end{multline*}
    By Proposition \ref{prop:bound-gap-Hodge-def} and the expression
    of $A_N$ given in \eqref{eq:ctrl-def-conn}, we obtain the desired control.

    For the second estimate, we need to commute $f$ with $\Box_N$ and
    $\Box_{N,\alpha}$. From the computations
    \begin{align*}
      [f,\Box_{N}]&=\Delta f+2\nabla f \cdot
                             \nabla^{L^p}\\
      [f,A_N]&=-2\sqrt{N}\langle \nabla f,\nabla \rho\rangle\\
      [f,B_N]&=0,
    \end{align*}
    and Proposition \ref{prop:bound-gap-Hodge-def} one has
    \begin{align*}
      \left\|\left[f,\left(\lambda-\frac
      1N\Box_{N,\alpha}\right)^{-1}\right]\right\|_{L^2\to
      L^2}&=\left\|\left(\lambda-\frac
            1N\Box_{N,\alpha}\right)^{-1}\left[f,\frac 1N\Delta_{N,\alpha}\right]\left(\lambda-\frac
            1N\Box_{N,\alpha}\right)^{-1}\right\|_{L^2\to L^2}\\
      &\leq C\left(N^{-1}\|\Delta
        f\|_{L^{\infty}}+N^{-\frac 12}\|\nabla
        f\|_{L^{\infty}}\|\nabla \rho\|_{L^{\infty}}+N^{-\frac
        12}\|\nabla f\|_{L^{\infty}}\right).
    \end{align*}

    Now
    \begin{align*}
      \exp(\alpha\sqrt{N}\rho)[f,S_N]\exp(-\alpha\sqrt{N}\rho)&=[f,\exp(\alpha\sqrt{N}\rho)S_N\exp(-\alpha\sqrt{N}\rho)]\\
      &=\frac{1}{2i\pi}\int_{|\lambda|=\mu_0}\left[f,\left(1-\frac{1}{N}\Box_{N,\alpha}\right)^{-1}\right],
    \end{align*}
    which concludes the proof.
  \end{proof}
In the case of a quantum spin system, $f$ is a finite sum of eigenfunctions of
$\Delta$, in which case the commutator is smaller.
\begin{lem}\label{lem:commut-2}
  Under the hypotheses of Lemma \ref{lem:commut}, if $\Delta
  f=-\mu f$, then there exists $C_2(\mu,M_0)$ such that
 \[
    \|\exp(\alpha\sqrt{N}\rho)[f,S_N]\exp(-\alpha\sqrt{N}\rho)\|_{L^2\to L^2}\leq CN^{-\frac 12}\|\nabla
        f\|_{L^{\infty}}\left(1+\|\nabla \rho\|_{L^{\infty}}\right).
    \]
  \end{lem}
  \begin{proof}
    The proof proceeds as previously, isolating $\Delta f = -\mu f$ in
    $[f,\Box_{N,\alpha}]$. A first application of the resolvent
    formula yields
    \begin{align*}
      \frac{1}{2i\pi}\int_{|\lambda|=\mu_0}&\left[f,\left(\lambda-\frac
      1N\Box_{N,\alpha}\right)^{-1}\right]
  \\
  &=\frac{1}{2i\pi}\int_{|\lambda|=\mu_0}\underbrace{\left(\lambda-\frac
      1N\Box_{N,\alpha}\right)^{-1}\left(\alpha N^{-\frac 12}\nabla f
      \cdot \nabla \rho+N^{-1}\nabla f \cdot
      \nabla^N\right)}_{\|\cdot\|_{L^2\to L^2}\leq CN^{-\frac
      12}\|\nabla f\|_{L^{\infty}}(1+\|\nabla \rho\|_{L^{\infty}})}\left(\lambda-\frac
      1N\Box_{N,\alpha}\right)^{-1}\\
    &+\frac{\mu N^{-1}f}{2i\pi}\underbrace{\int_{|\lambda|=\mu_0}\left(\lambda-\frac
        1N\Box_{N,\alpha}\right)^{-2}}_{=0}\\
    &+\mu N^{-1}\frac{1}{2i\pi}\int_{|\lambda|=\mu_0}\left[\left(\lambda-\frac
      1N\Box_{N,\alpha}\right)^{-1},f\right]\left(\lambda-\frac
    1N\Box_{N,\alpha}\right)^{-1}.
\end{align*}
By induction,
\[
\left\|\frac{1}{2i\pi}\int_{|\lambda|=\mu_0}\left[f,\left(\lambda-\frac
      1N\Box_{N,\alpha}\right)^{-1}\right]\right\|_{L^2\to L^2}\leq CN^{-\frac
      12}\|\nabla f\|_{L^{\infty}}(1+\|\nabla
    \rho\|_{L^{\infty}})\sum_{k=0}^{+\infty}(\mu N^{-1})^k.
  \]
  This concludes the proof.
\end{proof}

  We are now in position to prove a weighted estimate on eigenfunctions.

  \begin{prop}\label{prop:decay-C2-uniform}
    Let $M_0$ be a compact Kähler manifold. There exists $C_3>0$ such
    that, for every $N\geq 1$, for every
    $f\in C^2(M_0^{d'},\R)$ and every $\rho\in \mathrm{Lip}(M_0^{d'})$, if
    $S_N$ denotes the Szeg\H{o} kernel on $M_0^{d'}$ and if $\lambda\in
    \R,u\in H_0(M_0^{d'},L^{\otimes N})$ are such that $S_N(fu)=\lambda
    u$, then
\[
    \int_{M_0^{d'}} e^{2\alpha\sqrt{N}
      \rho(x)}(f(x)-\lambda-C(f,\rho)|\alpha|)|u(x)|^2\dd\text{Vol}(x)\leq 0,
  \]
  where
  \[
    C(f,\rho)=C_3\left[N^{-\frac 12}\|\Delta
        \rho\|_{L^{\infty}}+\|\nabla\rho\|_{L^{\infty}}\right]\left[N^{-1}\|\Delta
        f\|_{L^{\infty}}+N^{-\frac 12}\|\nabla
        f\|_{L^{\infty}}\left(1+\|\nabla
          \rho\|_{L^{\infty}}\right)\right].
    \]
    If $f$ is a sum of eigenfunctions of $-\Delta$ on $M_0^{d'}$, with
    frequencies bounded by $\mu$ independently on $d'$, then one can
    choose
    \[
    C(f,\rho)=C_3(\mu)\left[N^{-\frac 12}\|\Delta
        \rho\|_{L^{\infty}}+\|\nabla\rho\|_{L^{\infty}}\right]\left[N^{-\frac 12}\|\nabla
        f\|_{L^{\infty}}\left(1+\|\nabla
          \rho\|_{L^{\infty}}\right)\right].
    \]
  \end{prop}
  \begin{proof}
    Up to replacing $f$ with $f-\lambda$, one has $\lambda=0$.
    
    As in \cite{kordyukov_semiclassical_2018} one has
    \begin{align*}
      \langle
      \exp(\alpha\sqrt{N}\rho)(f-\lambda)u,\exp(\alpha\sqrt{N}\rho)u\rangle&=\langle
                                                     [S_N,\exp(2\alpha\sqrt{N}\rho)]fu,u\rangle\\
      &=\langle [S_N,\exp(2\alpha\sqrt{N}\rho)]S_Nfu,u\rangle+\langle
        [S_N,\exp(2\alpha\sqrt{N}\rho)](1-S_N)fu,u\rangle\\
      &=\langle [S_N,\exp(2\alpha\sqrt{N}\rho)][f,S_N]u,u\rangle.
    \end{align*}
    We write
    \[
      \langle [S_N,\exp(2\alpha\sqrt{N}\rho)][f,S_N]u,u\rangle=\langle
      e^{-\alpha\sqrt{N}\rho}[S_N,e^{2\alpha\sqrt{N}\rho}]e^{-\alpha\sqrt{N}\rho}e^{\alpha\sqrt{N}\rho}[f,S_N]e^{-\alpha\sqrt{N}\rho}e^{\alpha\sqrt{N}\rho}u,e^{\alpha\sqrt{N}\rho}u\rangle
    \]
    so that, by Lemma \ref{lem:commut},
    \begin{multline*}
      \left|\langle [S_N,\exp(2\alpha\sqrt{N}\rho)][f,S_N]u,u\rangle\right|\\ \leq C|\alpha| \left[N^{-1}\|\Delta
        f\|_{L^{\infty}}+N^{-\frac 12}\|\nabla
        f\|_{L^{\infty}}\left(1+\|\nabla \rho\|_{L^{\infty}}\right)\right]\left[N^{-\frac 12}\|\Delta
      \rho\|_{L^{\infty}}+\|\nabla
      \rho\|_{L^{\infty}}\right]\|\exp(\alpha\sqrt{N}\rho)u\|_2^2.
  \end{multline*}
  This concludes the proof in the general case.
  
  If $f$ is a sum of eigenfunctions of $-\Delta$, then one can remove
  the factor $N^{-1}\|\Delta f\|_{L^{\infty}}$ by Lemma \ref{lem:commut-2}.
  This concludes the proof.
  \end{proof}

\section{Case study: spin systems}
\label{sec:case-study:-spin}

In this section, we study Proposition \ref{prop:decay-C2-uniform} in
the particular case of spin systems.

\subsection{Construction of the weight}
\label{sec:construction-weight}

Let us construct a weight $\rho$ adapted to Proposition \ref{prop:decay-C2-uniform}.

 Let $U\subset M=(\S^2)^d$ be an open set. Let $\rho_0:M\to \R$ be as follows:
  \[
    \rho_0:x\mapsto 
    \begin{cases}
      0&\text{ if } \dist(x,U)\leq c_0\sqrt{d}\\
      \dist(x,U)-c_0\sqrt{d}&\text{ otherwise.}
    \end{cases}
    \]
    Let also $\chi:\R\to \R$ be as follows:
    \[
      \chi:x\mapsto
      \begin{cases}
        1-x^2&\text{ if } |x|<1\\
        0&\text{ otherwise}.
      \end{cases}
    \]
    We will inject in Proposition \ref{prop:decay-C2-uniform} the following function:
    \[
      \rho:x\mapsto \left[\int_{y\in
        M}\chi\left(\frac{2\dist(y,x)}{c_0\sqrt{d}}\right)\dd y\right]^{-1}\int_{y\in
        M}\chi\left(\frac{2\dist(y,x)}{c_0\sqrt{d}}\right)\rho_0(y)\dd y.
    \]
    Note that $\rho$ is supported on $\{\dist(x,U)\geq
    \frac{c_0}{2}\sqrt{d}\}$ and is greater than $\frac 12 \dist(x,U)$
    on $\{\dist(x,U)\geq 3c_0\sqrt{d}\}$.     
    \begin{prop}
The following controls hold independently on $c_0$ and $d$:
    \begin{align*}
      \|\nabla \rho\|_{L^{\infty}}&\leq 1\\
      \|\Delta \rho\|_{L^{\infty}}&\leq \frac{16\sqrt{d}}{c_0}.
    \end{align*}
  \end{prop}
  \begin{proof}
    Let $x_0,x_1\in M$. There exists $u\in (\mathfrak{so}_3)^d$, of
    norm $1$, such that $\exp(\dist(x_0,x_1)u)x_0=x_1$. From the
    definition of $\rho$ and the invariance of the integral kernel
    under $(SO_3)^d$, one has
    \[
      \rho(x_0)-\rho(x_1)=\left[\int_{y\in
        M}\chi\left(\frac{2\dist(y,x)}{c_0\sqrt{d}}\right)\dd
      y\right]^{-1}\int_{y\in
      M}\chi\left(\frac{2\dist(y,x)}{c_0\sqrt{d}}\right)(\rho_0(y)-\rho_0(\exp(\dist(x_0,x_1)u)y)\dd
    y.
  \]
  From there, since $\rho_0$ is $1$-Lipschitz, $\rho$ is
  $1$-Lipschitz.

  To estimate $\Delta \rho$, let us bound
  \[
    \int_{y\in
        M}\left|\nabla_x\chi\left(\frac{2\dist(y,x)}{c_0\sqrt{d}}\right)\right|\dd
      y,\]
    where the norm of the gradient is the $\ell^2$ norm.

    First, one has almost everywhere
    \[
      \nabla_x\chi\left(\frac{2\dist(y,x)}{c_0\sqrt{d}}\right)=\frac{8}{c_0^2d}\dist(x,y)\mathds{1}_{d(x,y)\leq
        \frac{c_0\sqrt{d}}{2}}\gamma,
    \]
    where $\gamma$ is the derivative at $0$ of the unique unit speed
    geodesic from $y$ to $x$ with minimal length.

    In particular, on the complement of $\left\{d(x,y)\in \left[\frac{c_0\sqrt{d}}{2}\left(\sqrt{1+\frac{1}{4d^2}}-\frac{1}{2d}\right),\frac{c_0\sqrt{d}}{2}\right]\right\}$
    one has
    \[\left\|\nabla_x\chi\left(\frac{2\dist(y,x)}{c_0\sqrt{d}}\right)\right\|\leq
      \frac{4\sqrt{d}}{c_0}\chi\left(\frac{2\dist(y,x)}{c_0\sqrt{d}}\right)
    \]
    so that
    \[
      \int_{B\left(x,\frac{c_0\sqrt{d}}{2}\left(\sqrt{1+\frac{1}{4d^2}}-\frac{1}{2d}\right)\right)}\left\|\nabla_x\chi\left(\frac{2\dist(y,x)}{c_0\sqrt{d}}\right)\right\|\dd
      y\leq
      \frac{4\sqrt{d}}{c_0}\int_{M}\chi\left(\frac{2\dist(y,x)}{c_0\sqrt{d}}\right)\dd
      y.
    \]
    To estimate the integral on the complement, we introduce
    $f:\R^+_*\to \R^+_*$ as the ratio of the area of spheres on $M$ versus $\C^d$:
    \[
      f:r\mapsto
      \cfrac{\mathop{Vol}_{2d-1}(S_M(x,r))}{\mathop{Vol}_{2d-1}(S_{\C^{d}}(0,r))}.
      \]
      An essential property of $f$ is that it is decreasing. Indeed,
      \[
        f(r)=\frac{1}{\mathop{Vol}_{2d-1}(S_{\C^{d}}(0,1))}\int_{S_{\C^d}(0,1)}\prod_{i=1}^d\mathds{1}_{|z_i|<\frac{\pi}{r}}\frac{\sin(r|z_i|)}{r|z_i|}\dd
        z \dd \overline{z}
      \]
      where the quantity to be integrated decreases with respect to $r$.

      Now
      \begin{multline*}
        \int_{B\left(x,\frac{c_0\sqrt{d}}{2}\right)\setminus B\left(x,\frac{c_0\sqrt{d}}{2}\left(\sqrt{1+\frac{1}{4d^2}}-\frac{1}{2d}\right)\right)}\left\|\nabla_x\chi\left(\frac{2\dist(y,x)}{c_0\sqrt{d}}\right)\right\|\dd
      y\\=\text{Vol}_{2d-1}(S^{2n-1})\bigintsss_{\frac{c_0\sqrt{d}}{2}\left(\sqrt{1+\frac{1}{4d^2}}-\frac{1}{2d}\right)}^{\frac{c_0\sqrt{d}}{2}}r^{2d-1}f(r)\frac{8}{c_0^2d}r\dd
      r.
    \end{multline*}
    
    For $r$ in the integration range, $f(r)\leq f\left(r\left(1-\frac
      1{2d}\right)\right)$ since $f$ is decreasing; moreover, for all
  $d\in \N$,
    \[
      \frac{r^{2d}}{\left(r\left(1-\frac
      1{2d}\right)\right)^{2d}}=\left(1-\frac{1}{2d}\right)^{-2d}\leq 4.      
    \]    
      Hence,
      \begin{multline*}
       \int_{B\left(x,\frac{c_0\sqrt{d}}{2}\right)\setminus B\left(x,\frac{c_0\sqrt{d}}{2}\left(\sqrt{1+\frac{1}{4d^2}}-\frac{1}{2d}\right)\right)}\left\|\nabla_x\chi\left(\frac{2\dist(y,x)}{c_0\sqrt{d}}\right)\right\|\dd
      y\\\leq
      4\text{Vol}_{2d-1}(S^{2n-1})\bigintsss_{\frac{c_0\sqrt{d}}{2}\left(\sqrt{1+\frac{1}{4d^2}}-\frac{1}{2d}\right)\left(1-\frac{1}{2d}\right)}^{\frac{c_0\sqrt{d}}{2}\left(1-\frac{1}{2d}\right)}r^{2d}f(r)\frac{8}{c_0^2d}\dd
      r.
      \end{multline*}
    Since
    \[
      1-\frac{1}{2d}\leq \sqrt{1+\frac{1}{4d^2}}-\frac{1}{2d},
    \]
    one is left with part of the integral controlled previously:
    \[
      \text{Vol}_{2d-1}(S^{2n-1})\bigintsss_{\frac{c_0\sqrt{d}}{2}\left(\sqrt{1+\frac{1}{4d^2}}-\frac{1}{2d}\right)\left(1-\frac{1}{2d}\right)}^{\frac{c_0\sqrt{d}}{2}\left(1-\frac{1}{2d}\right)}r^{2d}f(r)\frac{8}{c_0^2d}\dd
      r\leq
      \frac{4\sqrt{d}}{c_0}\int_{M}\chi\left(\frac{2\dist(y,x)}{c_0\sqrt{d}}\right)\dd
      y.
    \]
    Thus, one has the following control:
\[
    \left(\int_{M}\chi\left(\frac{2\dist(y,x)}{c_0\sqrt{d}}\right)\dd
      y\right)^{-1}\int_M\left\|\nabla_x\chi\left(\frac{2\dist(y,x)}{c_0\sqrt{d}}\right)\right\|\dd
    y\leq \frac{16}{c_0}\sqrt{d}.
  \]
  Let $x\in M$. Without loss of generality, $x=(1,\ldots, 1)$ is the
  North pole. Let $(X_i)_{1\leq i\leq d}$ and $(Y_i)_{1\leq i\leq d}$
  be the vector fields on $M$ corresponding to unit speed rotation
  around the $X$ or $Y$ axis on the $i$-th sphere. Then
  \[
    \Delta
    \rho(x)=\left(\int_{M}\chi\left(\frac{2\dist(y,x)}{c_0\sqrt{d}}\right)\dd
    y\right)^{-1}\sum_{i=1}^d\int_M\partial_{X_i}\chi\left(\frac{2\dist(y,x)}{c_0\sqrt{d}}\right)\partial_{X_i}\rho_0(y)+\partial_{Y_i}\chi\left(\frac{2\dist(y,x)}{c_0\sqrt{d}}\right)\partial_{Y_i}\rho_0(y).
  \]
  The semidefinite scalar product induced by $(X_i)$ and $(Y_i)$ is
  everywhere controlled by the usual one: for all $u,v\in TM$ with
  same base point,
  \[
    \left|\sum_{i=1}^d\langle X_i,u\rangle\langle X_i,v\rangle + \langle
      Y_i,u\rangle \langle Y_i,v\rangle\right| \leq \left|\langle u,v\rangle\right|.
  \]
  In particular, since $\|\nabla \rho_0\|\leq 1$, one has
\[
|\Delta \rho(x)|\leq
\left(\int_{M}\chi\left(\frac{2\dist(y,x)}{c_0\sqrt{d}}\right)\dd y \right)^{-1}\int_M\left\|\nabla_x\chi\left(\frac{2\dist(y,x)}{c_0\sqrt{d}}\right)\right\|\dd
  y,
\]
which concludes the proof.
\end{proof}

\subsection{Implementing the weighted estimates}
\label{sec:impl-weight-estim}

To begin with, let us define the class of symbols, called \emph{tame
  spin systems}, with which we will
work.

\begin{defn}\label{def:tame-spin-system}
Let $G=(V,E)$ be a graph with $|V|=d$ vertices. Suppose that the
valence at each site is bounded by $v$. Assign to each edge $e\in
E$ a colour among $k$ elements; one can decompose 
$E=E_1\sqcup E_2\sqcup \ldots \sqcup E_k$ into the disjoint union of
the sets of edges of a prescribed colour. Now, for each colour $j$,
let $w_j:M_0\times M_0\to \R$ be a $C^2$ function, where
$M_0=(\S^2)^{m_0}$ is a product of spheres; suppose that $w_j$ is a
finite sum of eigenfunctions of the Laplace operator.

Then the following function $g$ is a {\bfseries tame spin system} on $(M_0)^G=\{(x_a),a\in V\}$:
\[
  g:x\mapsto \sum_{j=1}^k\sum_{(a,b)\in E_j}w_j(x_a,x_b).
\]
\end{defn}

This very broad class of functions contains any finite-range spin
system on a lattice, quasi-crystal, or random graph with bounded valence, with any reasonable boundary
condition. Examples of spin systems not satisfying the control above are
\begin{itemize}
\item The boundary condition ``all spins at the boundary are
  identical'', except for spin chains
\item Infinite range interactions (with sufficiently slow decay)
\item Mean field theories
\item Random interactions (if the strength of the interaction is not
  bounded).
\end{itemize}

Since this section is concerned with the $d\to +\infty$ limit, we will
consider $d$-dependent families of tame spin systems. Without risk of
confusion, we will call ``tame spin system'' a family of tame spin
systems where,
with the notations of Definition \ref{def:tame-spin-system}, the objects
$m_0,v,k,(w_j)_{1\leq j \leq k}$ are fixed. 

The following property follows immediately from the definition.
\begin{prop}\label{prop:bound-der-tame-spin}
  Let $g$ be a tame spin system. There exists $C$ such that, for every
  $d$, one has
  \begin{align*}
    \|g\|_{L^{\infty}}&\leq Cd\\
    \|\nabla g\|_{L^{\infty}}&\leq C\sqrt{d}\\
    \|\Delta g\|_{L^{\infty}}&\leq Cd.
  \end{align*}
\end{prop}
We will not apply Proposition \ref{prop:decay-C2-uniform} to a tame
spin system $g$
itself, but to the $N$-dependent symbol $f$ which is such that
\[
  T_N(f)=d^{-1}T_N(g-\lambda)^2,
\]
where $\lambda$ is the eigenvalue to be studied.

The properties of $f$ depend on the symbol calculus on $\S^2$.
\begin{prop}
  Uniformly in $N$ and $d$, one has
  \begin{align*}
  \|\nabla f\|_{\infty}&\leq C\sqrt{d}\\
  \|\Delta f\|_{\infty}&\leq Cd\\
  f&=d^{-1}(u-\lambda)^2+O(N^{-1}).\\
\end{align*}
\end{prop}
\begin{proof}
  For $N\in \N$, let $\mathcal{B}_N$ denote the Berezin transform,
  defined as follows: for $f\in C^{\infty}(M,\R)$, the operator
  $T_N(f)$ has an integral kernel; we let
  \[
    \mathcal{B}_Nf:x\mapsto \frac{\pi}{N+1}T_N(f)(x,x).
  \]
  The Berezin transform is related to the symbol product
  (\cite{charles_berezin-toeplitz_2003}, Proposition 6). It admits an
  expansion in negative powers of $N$:
  \[
    \mathcal{B}_N=\mathrm{I}+\sum_{k=1}^{+\infty}N^{-k}B_k+O(N^{-\infty}),
  \]
  where $B_k$ is a differential operator of order $2k$.

  The operator $\mathcal{B}_N$ commutes with the $SO(3)$ action on
  $\S^2$ (since the Szeg\H{o} kernel is invariant by this action). In
  particular, there exist coefficients $(c_{\ell,k})_{\ell\leq k}$ such
  that, for every $k$,
  \[
    B_k=\sum_{\ell=0}^kc_{\ell,k}\Delta^{\ell}.
  \]
  Moreover, one has $\mathcal{B}_N(1)=1$ by definition, so that
  $c_{0,k}=0$ for all $k\geq 1$. In other terms, for some differential
  operators $C_k$ one can write
  \[
    \mathcal{B}_N=\mathrm{I}+\sum_{k=1}^{+\infty}N^{-k}C_k\Delta+O(N^{-\infty}).
  \]
  The symbolic product is then a polarisation of the Berezin
  transform: a monomial term in $\mathcal{B}_N$ of the form
  $\Delta^{\ell}$ leads to a term in the symbol product of $a$ and $b$ of the form
  $\partial^{\ell}a \overline{\partial}^{\ell}b$.

  The Berezin transform on $(\S^2)^d$ is the tensor product of the
  Berezin transform on each sphere. In particular, one has
  \[
    df=(g-\lambda)^2+\sum_{\substack{\mathcal{J}\subset
        \{1,\ldots,d\}\\|\mathcal{J}|\geq 1}}\prod_{j\in \mathcal{J}}\left(\sum_{k=1}^{+\infty}N^{-k}\widetilde{C_{k;j}}\right)\left(\partial^{\mathcal{J}}(g-\lambda),\overline{\partial}^{\mathcal{J}}(g-\lambda)\right).
  \]
  Here, $\widetilde{C_{k;j}}$ denotes the polarisation of $C_k$ acting
  on the $j$-th coordinate (holomorphic derivatives act on the first
  function, antiholomorphic derivatives on the second function). We,
  crucially, use the fact that the Berezin transform, and the symbol
  calculus, lead to absolutely converging sums for spin systems.

  If $g$ is a tame spin system, then for any $j_0\in \{1,\ldots,d\}$
  the number of $\mathcal{J}\subset \{1,\ldots,d\}$
  such that $j_0\in \mathcal{J}$ and $\partial^{\mathcal{J}}g\neq 0$ is
  bounded independently on $j_0$ and $d$. Using the notations of Definition
  \ref{def:tame-spin-system}, an upper bound is $2^{m_0v}-1$. In particular, uniformly in $j_0$
  and $d$,
  \[
    \sum_{\substack{\mathcal{J}\subset
        \{1,\ldots,d\}\\j_0\subset \mathcal{J}}}\prod_{j\in
      \mathcal{J}}\left(\sum_{j=1}^{+\infty}N^{-k}\widetilde{C_{k;j}}\right)\left(\partial^{\mathcal{J}}(g-\lambda),\overline{\partial}^{\mathcal{J}}(g-\lambda)\right)=O(N^{-1}).
  \]
  In fact, $N(df-(g-\lambda)^2)$ is again a tame spin system (with
  classical dependence on $N$) and
  satisfies the same type of bounds as $g$, as in Proposition
  \ref{prop:bound-der-tame-spin}. This yields the desired bounds on
  $\nabla f$ and $\Delta f$.
\end{proof}

{\bfseries Proof of Theorem \ref{thr:weighted-spins}}

Let $\rho$ be constructed as in Section
\ref{sec:construction-weight} (we will define $U$ and $c_0$ later) and
let $f$ be as above. The spectral gap condition of
Proposition \ref{prop:bound-gap-Hodge-def} amounts (for $d$ large
enough) to
\[
  |\alpha|\leq c_3N^{\frac 12}d^{-\frac 12}c_0
\]
and the constant in Proposition \ref{prop:decay-C2-uniform} is
controlled by
\[
  C(f,\rho)\leq C_3(\mu)N^{-1}dc_0^{-1}.
\]
In particular, one has

  \[
    \int_{M} e^{c_4\sqrt{N}\alpha \rho}\left(f-C(\mu)N^{-\frac
        12}d^{\frac 12}\right)|u|^2\leq 0.
  \]
  Let now $\lambda_1=C(\mu)N^{-\frac 12}\sqrt{d}$ and
  $\lambda_2=2C(\mu)N^{-\frac 12}\sqrt{d}$. One has
  $(h-Cc_3(\lambda+N^{-\frac 12}d))\geq C(\mu) N^{-\frac
    12}\sqrt{d}$ on $\{h\geq \lambda_2\}$.

  We now choose
  \begin{align*}
    U&=\{f\leq \lambda_1\}\\
    c_0&=N^{-\frac 12}
  \end{align*}
  and $\alpha$ large enough.
  Then, decomposing the integral yields, for some $c_4>0$,
  \begin{multline*}
    0\geq C(\mu) N^{-\frac
    12}\sqrt{d} \int_{\left\{x\in M,\dist(x,U)\geq 3N^{-\frac
          12}\sqrt{d}\right\}\cap \{f\geq 2C(\mu)N^{-\frac
        12}\sqrt{d}\}}e^{c_4\sqrt{N}\dist(x,U)d^{-\frac 12}}|u|^2(x)\dd
    x\\
    -C(\mu)N^{-\frac 12}\sqrt{d}\int_{\{x\in U\}}|u|^2(x)\dd x.
  \end{multline*}
  Since $f=\frac 1d (g-\lambda)^2+O(N^{-1})$, one has
  \[
    \{f(x)\geq \lambda_1\} \subset \{|g(x)-\lambda|\geq CN^{-\frac
      14}d^{\frac 34}\}.
  \]

  This yields Theorem \ref{thr:weighted-spins}.

  \begin{rem}

  The window $|g-\lambda|\geq CN^{-\frac 14}d^{\frac 34}$ seems
  larger than what Theorem \ref{thr:conc-Lip} allows for: by applying
  Proposition \ref{prop:decay-C2-uniform} directly to $g$, we would
  have obtained $|g-\lambda|\geq CN^{-\frac 12}d^{\frac 12}$. However,
  since even the lowest eigenvalue of $T_N(g)$ is typically of order
  $N^{-1}d$ if $\min(g)=0$, this constant appears in the lower bound for the
  negative part of the weighted integral; this would yield an estimate
  of the form
  \[
    \int_We^{c_4\sqrt{N}\dist(x,U)d^{-\frac 12}}|u|^2(x)\dd x \leq
    CN^{-\frac 12}\sqrt{d},
  \]
  which, as $d$ increases, is no better than the trivial estimate
  \[
    \int_We^{c_4\sqrt{N}\dist(x,U)d^{-\frac 12}}|u|^2(x)\dd x\leq
    Ce^{C\sqrt{N}}.
  \]
\end{rem}

\begin{rem}\label{rem:modif-Thm-D}
  Letting $c_0$ be a small constant rather than $N^{-\frac 12}$ in the
  proof of Theorem \ref{thr:weighted-spins}, we obtain the following variant:
  \[
    \int_We^{c_4N\dist(x,U)d^{-\frac 12}}|u(x)|^2\dd x \leq C,
  \]
  where
  \[
    W=\{x\in (\S^2)^d, \dist(x,U)\geq Cc_0\sqrt{d}\}.
  \]
\end{rem}

\subsection{The partition function}
\label{sec:weyl-law}
To conclude, in this subsection we use Theorem \ref{thr:weighted-spins} to prove
Proposition \ref{prop:control-free-energy}.

We let $g$ be a tame spin system and $\beta\leq cN^{\frac 12}d^{-1}$. Let $(u_k)_{1\leq k
  \leq (N+1)^d}$ be a spectral basis for $T_N(g)$ and
$(\lambda_k)_{1\leq k \leq (N+1)^d}$ be the associated family of eigenfunctions. Then
\[
  Tr(e^{-\beta T_N(g)})=\sum_{k=1}^{(N+1)^d}e^{-\beta \lambda_k}.
\]
We wish to compare
\[
  \langle u_k,e^{-\beta g}u_k\rangle
\]
and
\[
  \langle u_k,e^{-\beta T_N(g)}u_k\rangle=e^{-\beta \lambda_k},
\]
for $\beta\leq cN^{\frac 12}d^{-1}$.

Following Theorem \ref{thr:weighted-spins}, let
\[
  W=\left\{x\in (\S^2)^d, \dist(x,\{|g-\lambda_k|<C_0N^{-\frac 14}d^{\frac
    34}\})>KC_0N^{-\frac 12}\sqrt{d}\right\}.
\]
Here $C_0$ is the constant $C$ in Theorem \ref{thr:weighted-spins}, and $K$ is an integer
large enough.

Since $\|\nabla g\|<C(w,v)\sqrt{d}$ by Proposition
\ref{prop:bound-der-tame-spin}, for $x\in (\S^2)^d\setminus \Omega$
one has $|g-\lambda_k|<CN^{-\frac 12}d$. In particular, for
some $C>0$ independent of $N$ and $d$, one has
\[
  e^{-\beta \lambda_k}e^{-CN^{-\frac 12}\beta d}\int_{W^c}|u(x)|^2\leq
  \int_{W^c}e^{-\beta g}|u(x)|^2\leq e^{-\beta \lambda_k}e^{CN^{-\frac
    12} \beta d}\int_{W^c}|u(x)|^2.
\]
A first application of Theorem \ref{thr:weighted-spins} yields
\[
  \int_W|u(x)|^2\leq C_0e^{-c_0K},
\]
so that, in particular, for $K$ large enough
\[
  \frac 12 \leq \int_{W^c}|u(x)|^2\leq 1.
\]
One can then simplify the previous inequality into
\[
  e^{-\beta \lambda_k}e^{-C'N^{-\frac 12}\beta d}\leq
  \int_{W^c}e^{-\beta g}|u(x)|^2\leq e^{-\beta
    \lambda_k}e^{CN^{-\frac 12}\beta d},
\]
for some $C'>C$.

It remains to give an upper bound on
\[
  \int_We^{-\beta g}|u|^2.
\]

To this end, we observe that, on $W$,
\[
  -\beta(g-\lambda_k)\leq \beta CN^{-\frac 14}d^{\frac
    34}+\dist(x,\{|g-\lambda|<CN^{-\frac 14}d^{\frac 34}\})\beta
  \|\nabla g\|.
\]
By Proposition \ref{prop:bound-der-tame-spin}, the bound on $\beta$
and the definition of $W$, this simplifies into
\begin{align*}
  -\beta(g-\lambda_k)&\leq C(N^{-\frac 12}\beta d^{-1}) N^{\frac 12}d^{-\frac
                       12}\dist(x,\{|g-\lambda|<CN^{-\frac 14}d^{\frac 34}\})\\
  &\leq c_0N^{-\frac 12}d^{-\frac
                       12}\dist(x,\{|g-\lambda|<CN^{-\frac 14}d^{\frac
    34}\})
\end{align*}
if $\beta<cN^{\frac 12}d^{-1}$ with $c$ small enough. Here $c_0$ is the exponential constant in
Theorem \ref{thr:weighted-spins}.

Hence, by Theorem \ref{thr:weighted-spins},
\[
  \int_We^{-\beta g}|u_k|^2\leq e^{-\beta
    \lambda_k}\leq e^{-\beta
    \lambda_k}(e^{C'N^{-\frac 12}\beta d}-1).
\]
To conclude,
\[
  \frac 12 e^{-\beta \lambda_k}e^{-C'N^{-\frac 12}\beta d}\leq \int e^{-\beta
    g}|u_k|^2\leq e^{-\beta \lambda_k}e^{C'N^{-\frac 12}\beta d}.
\]
Summing this estimate over $k$ yields:
\[
  Tr(\exp(-\beta T_N(g)))e^{-C'N^{-\frac 12}\beta d}\leq \int e^{-\beta
    g}\left(\sum_{k}|u_k|^2\right)\leq Tr(\exp(-\beta
  T_N(g)))e^{C'N^{-\frac 12}\beta d}.
\]
Since the $u_k$'s form an orthonormal basis of the Hilbert space, one
has, for every $x\in M$,
\[
  \sum_{k}|u_k|^2=\left(\frac{N+1}{\pi}\right)^d.
\]
Up to this factor, the quantum partition function $ Tr(\exp(-\beta
T_N(g)))$ is, approximately, given by the classical partition function
$\int_Me^{-\beta g}$. This concludes the proof.

\begin{rem}
The multiplicative error term $e^{C'N^{\delta-\epsilon}}$ is better
than the outcome of the method used by Lieb
\cite{lieb_classical_1973}. In this method, one bounds the quantum
partition function by a term of the form
\[
  \left(\frac{N+1+2C}{\pi}\right)^d\int_Me^{-\beta g},
\]
where $C$ is the maximal order of the spin polynomial at one given
site (the order of $S_{x,j}S_{x,j+1}$ is $1$, but the order of
$S_{x,j}S_{y,j}$ is $2$). The error is then
\[
  \left(\frac{N+1+2C}{N+1}\right)^d=O(\exp(CdN^{-1})).
\]
\end{rem}

\begin{rem}
  In this method, the upper bound on $\beta$ is driven by the bound in
  the weighted estimate in Theorem \ref{thr:weighted-spins}. Following Remark
  \ref{rem:modif-Thm-D}, on the improved range $\beta\leq cNd^{-1}$,
  we obtain the weaker estimate
  \[
    |f_Q(N,\beta)-f_C(N,\beta)|\leq C.
    \]
\end{rem}

%%% Local Variables:
%%% mode: latex
%%% TeX-master: "Article.tex"
%%% End:

\section{Acknowledgements}
\label{sec:acknowledgements}

This material is based upon work supported by the National Science
Foundation under Grant No. DMS-1440140 while A. Deleporte was in
residence at the Mathematical Sciences Research Institute in Berkeley,
California, during the Fall 2019 semester.

\bibliographystyle{abbrv}
\bibliography{main}
\end{document}